\documentclass[12pt]{amsart}
\usepackage{enumerate}
\usepackage{graphicx,graphics}
\usepackage{soul}
\usepackage{amsfonts}
\usepackage{amssymb}
\usepackage{amsthm}
\usepackage{amsmath}
\usepackage{marginnote}
\usepackage{mathrsfs,color}
\usepackage{tikz} 
\usepackage{hyperref}
\input{xy}
\xyoption{all}

\numberwithin{equation}{section}
\newtheorem{theorem}{Theorem}[section]
\newtheorem{lemma}[theorem]{Lemma}

\newtheorem{remark}[theorem]{Remark}
\newtheorem{corollary}[theorem]{Corollary}

\newtheorem{proposition}[theorem]{Proposition}

\newtheorem{definition}[theorem]{Definition}

\newtheorem{fact}[theorem]{\textbf{Fact}}
\newcommand*{\QEDA}{\hfill\ensuremath{\blacksquare}}
\DeclareMathOperator{\Int}{Int}
\DeclareMathOperator{\cco}{Conv}
\DeclareMathOperator{\Aff}{Aff}
\DeclareMathOperator{\Gr}{\mathbf{G}}
\DeclareMathOperator{\ag}{\mathbf{F}}
\DeclareMathOperator{\im}{Im}

\begin{document}

\title[The hyperspace of $k$-dimensional closed convex sets]{The hyperspace of $k$-dimensional closed convex sets}

\author[A. Escobedo-Bustamente and  N. Jonard-P\'erez]{Adriana Escobedo-Bustamante and   Natalia Jonard-P\'erez}

\address{Adriana Escobedo-Bustamante\\ 
              Facultad de Ciencias Exactas, Universidad Ju\'arez del Estado de Durango, 34113. Durango, Dgo. M\'exico.}
            \email{adriana.escobedo@ujed.mx}

\address{Natalia Jonard-P\'erez\\ 
Departamento de  Matem\'aticas,
Facultad de Ciencias, Universidad Nacional Aut\'onoma de M\'exico, 04510 Ciudad de M\'exico, M\'exico.
              Mexico.}
\email{ nat@ciencias.unam.mx}



\thanks{This work has been supported by  PAPIIT grant IN101622  (UNAM, M\'exico).}


\maketitle

\begin{abstract} 
For every $n \geq 2$, let $\mathcal K_k^n$ denote the hyperspace of all $k$-dimensional closed convex subsets of the Euclidean space $\mathbb R^n$ endowed with the Atouch-Wets topology.  Let $\mathcal K_{k,b}^n$ be the subset of $\mathcal K_k^n$ consisting of all $k$-dimensional compact convex subsets. In this paper we explore the topology of  $\mathcal K_k^n$ and $\mathcal K_{k,b}^n$ and the relation of these hyperspaces with the Grassmann manifold $\Gr_k(n)$. We prove that both $\mathcal K_k^n$ and $\mathcal K_{k,b}^n$ are Hilbert cube manifolds with a fiber bundle structure over $\Gr_k(n)$. We also show that the fiber of $\mathcal K_{k,b}^n$ with respect to this fiber bundle structure is homeomorphic with $\mathbb R^{\frac{k(k+1)+2n}{2}}\times Q$, where $Q$ stands for the Hilbert cube. 
\end{abstract}

\section{Introduction}\label{sec:intro}

Let $H$ denote a Hilbert space and let $\mathcal K^H$ be the family of all closed convex subsets of $H$. There are several ways to equip $\mathcal K^H$ with an interesting topology. The Vietoris topology, the Fell topology, the Wijsman topology, the Atouch-Wets metric or the Hausdorff metric are classic examples.  We refer the reader to \cite{{Beer1993}} for a detailed study of these and other topologies.  In the last 50 years, the topological type of several subspaces of $\mathcal K^H$ (equipped with some of the mentioned topologies) have been strongly studied. For instance, in the seminal work \cite{Nadler}, the authors proved that the family of all compact convex subsets of a compact convex set $K$ with $\dim K\geq 2$ (equipped with the Hausdorff metric topology) is homeomorphic with the Hilbert cube, among other related results. 

Other hyperspaces of bounded convex sets have been studied in \cite{AntonyanNatalia} \cite{AntonyanNataliaSaulwidth}, \cite{Bazylevych-93}, \cite{Bazylevych-97}, \cite{Bazylevych Zarichnyi-2006}, \cite{JonardMerino}, \cite{Luisa-Nat1}, \cite{Sakai Bounded}.
For non-bounded convex sets, we refer the reader to \cite{SakaiYaguchi2006}, \cite{Sakai-Yang-2007} and \cite{Luisa-Nat}.
In particular, in \cite{Sakai-Yang-2007} the authors proved that if $H=\mathbb R^n$, then the hyperspace $\mathcal K^H$ (equipped with the Fell topology) is homeomorphic with $Q\times\mathbb R^n$.

In this bestiary, the most difficult subspaces of $\mathcal K^H$ to deal with are those that are neither open nor closed in $\mathcal K^H$ (see, e.g., \cite{Belegradek} and \cite{Hetman}). In this work, we are interested in one of those subspaces: the hyperspace of all $k$-dimensional closed convex subsets of $\mathbb R^{n}$  ($0\leq k\leq n$).

More precisely, if we denote by $\mathcal K_k^n$ and $\mathcal K_{k,b}^n$ the hyperspace of all $k$-dimensional closed convex subsets of $\mathbb R^n$ (equipped with the Attouch-Wets metric topology) and the hyperspace of all $k$-dimensional compact convex subsets of $\mathbb R^n$ (equipped with the Hausdorff metric topology), respectively, we will prove that both hyperspaces have a fiber bundle structure with  the Grassmann manifold $\Gr_k(n)$  as base (Theorem~\ref{t:main K_kn es haz fibrado}). As a corollary, we will prove that  $\mathcal K_k^n$ and 
$\mathcal K_{k,b}^n$ are $Q$-manifolds (corollaries~\ref{c:K_kn es Q manifold} and \ref{c:K_kbn es Q manifold}). In the particular case of $\mathcal K_{k,b}^n$ we also show that the fiber bundle structure has fiber homeomorphic with $\mathbb R^{\frac{k(k+1)+2n}{2}}\times Q$ (Corollary~\ref{c:K_kbn es Q manifold}).

Even if our main results are for the case when $H=\mathbb R^n$, we will also provide some results which are valid for the general case. 

 The paper is organized as follows. In Section~\ref{sec:prelim} we establish the notation we will use as well as some basic results related with the topology of the different families of $\mathcal K^H$. We also recall some classic results from the theory of $Q$-manifolds and convex geometry that will be used later. 
 One of our main tools will be the map $p:\mathcal K^H\to H$ that assigns to each closed convex set $A\in \mathcal K^H$, the point in $A$ closest to the origin. 
  In \cite{Sakai-Yang-2007}, the authors proved that the map
$p:\mathcal K^{H}\to  H$ is continuous (provided that $H=\mathbb R^{n}$ and $\mathcal K^{H}$ is equipped with the Fell topology). Unfortunately, there is a gap in their proof. Indeed, the reader can verify that the proof presented in \cite[Lemma 2]{Sakai-Yang-2007} only shows that the map which assigns to any element $A\in \mathcal K^n$ the value $\|p(A)\|$, is continuous. 
To fulfill this gap, we present in Lemma~\ref{l:continuidad funcion p} a correct proof that in fact generalizes the continuity of $p$ for the case where $H$ is any Hilbert space (not necessarily finite-dimensional) and $\mathcal K^H$ is endowed with the Attouch-Wets metric. Recall that in the finite-dimensional case, the Fell topology and the topology induced by the Attouch-Wets metric coincide in $\mathcal K^n$ (see Fact~\ref{f: AW=FELL}). 

The other main ingredient of this work is the Grassmann manifold $\Gr_k(H)$.  Namely, the family of all $k$-dimensional linear subspaces of $H$. These are very well-known objects which appear frequently in the differential realm and they are usually constructed as a quotient space of a manifold. However, the aim of this work is more synthetic and we are interested in seeing $\Gr_k(H)$ as hyperspaces (subspaces of $\mathcal K_k^H$).
This is why in Section~\ref{s:Grassmann} we will make a summary of all different (and equivalent) ways to equip $\Gr_k(H)$ with a hyperspace topology.
We will also explore the case of the family $\ag_k(H)$ consisting of all $k$-dimensional flats. 
We believe that some of the results presented in Section~\ref{s:Grassmann} are well-known folklores. However, we did not find formal proofs of them in the literature.  Since they are important for our main results, we decided to provide complete proofs for all of them. We also believe that those theorems can be of general interest.

In Section~\ref{s: fiber bundle of grasmanians} we prove several key lemmas that will be used later, and finally in Section~\ref{s:main} we prove our main results related to the topology of the hyperspaces $\mathcal K_k^n$ and $\mathcal K_{k,b}^n$.

\section{Preliminaries}
\label{sec:prelim}

We begin by introducing the notation and basic results that we will use throughout the work. The letter $H$ will always denote a real Hilbert space (however, some results are also valid for any normed linear space). When $H$ has finite dimension $n$, we will use the standard notation $\mathbb{R}^n.$ In any case, the inner product will be denoted by $\langle\cdot,\cdot\rangle.$ The corresponding norm and closed unit ball are denoted by $\|\cdot\|$ and $\mathbb{B},$ respectively. 
If $A\subset H$ is an arbitrary subset, the \textit{convex hull} of $A$ is denoted by $\cco(A)$. Similarly, we denote by $\Aff(A)$ the \textit{affine hull} of $A$. Namely,
$$\Aff(A):=\left\{\sum_{i=1}^m\lambda_ia_i: \lambda_i\in\mathbb R,\; a_i\in A, \;\sum_{i=1}^{m}\lambda_i=1, \; m\in\mathbb N \right\}.$$
A set $F\subset H$ is called a \textit{flat} iff $\Aff(F)=F$. 
Notice that $F$ is a flat if and only if $F-a$ is a vector subspace for every $a\in F$. 
We say that $F$ is a $k$-dimensional flat ($k\in\mathbb N)$ if $F-a$ is a $k$-dimensional vector subspace for any $a\in F$. Notice that in this case, $F$ is always closed in $H$.

The following easy remark will be used several times throughout this work.

\begin{remark}\label{r:distancia entre combinaciones convexas}
Let $a=\sum_{i=0}^k\lambda_ia_i$ and $b=\sum_{i=0}^k\lambda_ib_i$, where $\sum_{i=0}^k\lambda_i=1$,  $\lambda_i\in [0,1]$ and $a_i, b_i\in H$ for every $i=0,\dots, k$. If $\|a_i-b_i\|<r$ for each $i=0,\dots, k$, then $\|a-b\|<r$.
\end{remark}
\begin{proof}
    Simply observe that
    \begin{align*}
        \|a-b\|&=\left\|\sum_{i=0}^k\lambda_ia_i-\sum_{i=0}^k\lambda_ib_i\right\|
        =\left\|\sum_{i=0}^k\lambda_i(a_i-b_i)\right\|\\
        &\leq \sum_{i=0}^k\lambda_i\|a_i-b_i\|<\sum_{i=0}^k\lambda_i r=r.
    \end{align*}
\end{proof}

\subsection{Hyperspaces of convex sets}\label{ss:hyperspaces}

As mentioned in the introduction, we use the symbol $\mathcal{K}^H$ to denote the family of all nonempty closed and convex subsets of $H$. The family of all elements in $\mathcal{K}^H$ which are bounded will be denoted by $\mathcal{K}_b^H$.
We say that an element $A\in\mathcal K^H$ has dimension $k$ iff the flat $\Aff(A)$ has dimension $k$. We denote by $\mathcal K_k^H$ the set of all elements $A\in\mathcal{K}^H$ with dimension $k$.

Following this logic, for every $k\in \mathbb N$, we define the family
$$\mathcal{K}_{k,b}^H:=\mathcal{K}_{k}^H\cap \mathcal{K}_{b}^H.$$

If $H=\mathbb R^n$, we simply write $\mathcal{K}^n$, $\mathcal{K}_b^n$, $\mathcal{K}_k^n$, and $\mathcal{K}_{k,b}^n$, respectively.

 On $\mathcal{K}^H$ one can consider different  topologies  (see, e.g. \cite{Beer1993}). In this work, we are interested in the Attouch-Wets topology $\tau_{AW}.$  This topology is determined by the Attouch-Wets metric $d_{AW}$ which can be  defined for  any $A,B\in\mathcal{K}^H$ as
\begin{equation}
\label{eq:dAW}
d_{AW}(A,B):=\sup_{j\in\mathbb{N}}\left\{\min\left\{\frac{1}{j},\sup_{\|x\|<j}|d(x,A)-d(x,B)|\right\}\right\},
\end{equation}
This definition is the one used in \cite{SakaiYaguchi2006} and  it is equivalent to the Attouch-Wets metric defined in \cite[Definition 3.1.2]{Beer1993}.

It is important to point out that the Attouch-Wets metric and the well-known Hausdorff metric $d_H$ generate the same topology on $\mathcal K^H_b$ (see, e.g., \cite[Theorem 3.2]{SakaiYaguchi2006}).

Recall that the \textit{Hausdorff metric} between any  pair of nonempty  closed subsets  $A, B\subset H$ is defined by any of the following equivalent expressions (see, e.g. \cite[\S 3.2]{Beer1993})
\begin{align*}
d_H(A,B)&=\max\left\{\sup_{x\in A}d(x,B), \sup_{x\in B}d(x,A)\right\}\\
&=\text{inf}\left\{\lambda>0:A\subseteq B+\lambda\mathbb{B},\ B\subseteq A+\lambda\mathbb{B}\right\}\\
&=\sup_{x\in H}|d(x, A)-d(x,B)|.
\end{align*}

The following lemma will be used several times throughout the paper. We omit its proof because it is a direct consequence of (\ref{eq:dAW}).

\begin{lemma}\label{lem:Lemma-Ananda}
    Let $\varepsilon>0$ and $j\in\mathbb N$ be such that $\varepsilon\leq\frac{1}{j}$. 
    \begin{enumerate}[\rm(1)]
        \item If $d_{AW}(A,B)<\varepsilon$, then $$\sup_{\|x\|<j}|d(x,A)-d(x,B)|<\varepsilon.$$
        \item If  $\varepsilon\in (\frac{1}{j+1},\frac{1}{j}]$ and $\sup\limits_{\|x\|<j}|d(x,A)-d(x,B)|<\varepsilon$, then  $$d_{AW}(A,B)<\varepsilon.$$
    \end{enumerate}
\end{lemma}

We will also need the following technical lemma whose proof can be found in \cite[Lemma 3.9]{DonJuanJonardPerezLopezPoo}.

\begin{lemma}\label{l:distancia igual a distancia en interseccion}
For any closed subset $C\subset H$ and any positive scalar $j>0$, if $L>2j+d(0, C)$,  then 
$$d(x,C)=d\big(x,C\cap B(0, L) \big)\;\text{ for every }\;x\in B\left(0,j\right).$$
\end{lemma}

If we restrict ourselves to the family of closed convex sets containing the origin, the Attouch-Wets distance can be approached by means of the Hausdorff metric, as the following lemma shows.

\begin{lemma}(\cite[Lemma 2.1]{Luisa-Nat})
\label{l:d_H-dAW}
Let $r>0$. If $z\in r\mathbb B$ and $A, K\in \mathcal K^H$ are two closed convex sets containing the origin of $H$, then the following equalities hold. 
\begin{enumerate}[\rm(1)]
\item  $d(z,A)=d(z, A\cap r\mathbb B)$.
\item $\sup\limits_{\|x\|<r}|d(x,A)-d(x,K)|=\sup\limits_{x\in r\mathbb B}|d(x,A)-d(x,K)|.$
\item $d_H\big(A\cap r\mathbb B,K\cap r\mathbb B\big)=\sup\limits_{x\in r\mathbb B}|d(x,A)-d(x,K)|$. 
\item $d_{AW}(A,K)=\sup\limits_{j\in\mathbb{N}}\left\{\min\left\{\frac{1}{j},d_H\big(A\cap j\mathbb B,K\cap j\mathbb B\big)\right\}\right\}.$  
\item For every integer $j\geq 1$ and every $\varepsilon \in \left(\frac{1}{j+1}, \frac{1}{j}\right]$,
$$d_{AW}(A,K)<\varepsilon\textit{ if and only if }d_H\big(A\cap j\mathbb B,K\cap j\mathbb B\big)<\varepsilon.$$
\end{enumerate}
\end{lemma}

Let $X$ be a topological space, and let $\mathcal H$ be a family of closed subsets. \textit{The Fell topology} on $\mathcal{H}$, denoted by $\tau_F$, is the one generated by the sets $$U_{\mathcal H}^-:=\{A\in\mathcal{H} : A\cap U\neq\emptyset\}\;\text{ and }\;(X\setminus C)_{\mathcal H}^+:=\{A\in\mathcal{H}:A\subset X\setminus C\},$$ where $U\subset X$ is open and $C\subset X$ is compact. 

If it is possible to deduce from the context who the family $\mathcal H$ is, we simply write $U^-$ and $(X\setminus C)^+$.

Throughout this work we will use the following useful fact (see, for instance, \cite[Exercise 5.1.10(b)]{Beer1993} or \cite[Remark 1]{Sakai-Yang-2007}). 

\begin{fact}\label{f: AW=FELL}
    In  $\mathcal{K}^n$,  the Attouch-Wets topology and the Fell topology are the same.
\end{fact}

Following the notation above, the topology in $\mathcal H$ generated by all sets $U^-$  (with $U\subset X$  open) will be called \textit{the lower Vietoris topology} in $\mathcal H$ and will be denoted by $\tau_{LV}$. This topology will be extensively used later in the paper.


Given a family of closed convex sets $\mathcal H\subset \mathcal K^H$, 
for every $\varepsilon>0$ and $a_1,\dots,a_m\in H$, we will use the following notation:
$$[a_1,\dots,a_m;\varepsilon]_{\mathcal H}:=\bigcap_{i=1}^m B(a_i,\varepsilon)_{\mathcal H}^-,$$
where $B(a,\varepsilon)$ stands for the open ball in $H$ centered at $a$ with radius $\varepsilon$.
Notice that 
$$[a_1,\dots,a_m;\varepsilon]_{\mathcal H}=\{A\in \mathcal H: A\cap B(a_i,\varepsilon)\neq\emptyset\text{ for all }i=1,\dots, m\}.$$

As usual, if we can infer from the context who the family $\mathcal H$ is, we simply write $[a_1,\dots,a_m;\varepsilon]$.

It is easy to prove that the family 
$$\mathcal B:=\{[a_1,\dots,a_m;\varepsilon]_{\mathcal H}:a_1,\dots,a_m\in H, \varepsilon>0, m\in\mathbb N\}$$
is a base for the lower Vietoris topology in $\mathcal H$.

\begin{proposition}\label{p: basic properties}  
The following operations are always continuous.
\begin{enumerate}[\rm(1)]
    \item The addition $(\mathcal K^H, \tau_{LV})\times (\mathcal K^H, \tau_{LV})\to(\mathcal K^H, \tau_{LV})$ given by $$(A, B)\to A+B:=\{a+b: a\in A, b\in B\}.$$
    \item The addition $(\mathcal K^n, \tau_{F})\times \mathbb R^{n}\to(\mathcal K^n, \tau_{F})$ given by
    $$(A, \omega)\to A+\omega:=\{a+\omega: a\in A\}.$$
    
\end{enumerate}
\end{proposition}

\begin{proof}
  Let $U\subset H$ be an open set. If $(A+B)\in U^{-}$,   we can find points $a\in A$ and $b\in B$ such that $a+b\in U$. By the continuity of the sum in $H$, there exist neighborhoods $V$ and $W$ of $a$ and $b$, respectively, such that
  $$x+y\in U\; \text{ for every }x\in V\text{ and } y\in W.$$
Therefore, if $(A', B')\in V^{-}\times W^{-}$, it follows immediately that $(A'+B')\in U^{-}$.
  This proves (1) and half of (2). 

  To finish the proof, let $C\subset \mathbb R^{n}$ be a compact subset,  $A\in \mathcal K^H$  and $\omega\in \mathbb R^{n}$ be such that $(A+\omega)\in (\mathbb R^{n}\setminus C)^{+}$.
  Then $A\subset \mathbb R^{n}\setminus (C-\omega)$ and since $C-\omega$ is compact and $A$ is closed, we can find $\eta>0$ such that
  $$\inf\{d(a,c-\omega):a\in A,\;c\in C\}>\eta.$$
  Let $\delta=\eta/3$ and observe that $D:=(C-\omega)+\delta\mathbb B^n$ is a compact subset of $\mathbb R^{n}$ such that $A\in \mathcal O:=(\mathbb R^{n}\setminus D)^{+}$. Hence, for every $(B,v)\in \mathcal O\times B(\omega, \delta)$, if $(B+v)\cap C\neq\emptyset$, we can find points $b\in B$ and $c\in C$ such that $b+v=c$. Therefore 
  $$b=(c-\omega)+(\omega-v)\in (C-\omega)+B(0,\delta)\subset D$$
  which directly implies that $B\cap D\neq\emptyset$, a contradiction. Thus, $(B+v)\in (\mathbb R^{n}\setminus C)^{+}$, and now the proof of (2) is complete.

\end{proof}

\subsection{Affinely independent families}

Let us recall that a set of points 
 $\{x_0, x_1, \dots, x_k \}\subset H$ is called an \textit{affinely independent set} (also called \textit{geometrically independent set} )   iff for every real scalar values $\alpha_0,  \dots, \alpha_k$ such that 
$$\sum_{i=0}^k \alpha_i x_i=0 \; \text{and} \; \sum_{i=0}^k \alpha_i=0,$$
it follows that $\alpha_0= \alpha_1= \cdots=\alpha_k=0$. Equivalently, $\{x_0, x_1, \dots, x_k \}$ is affinely independent if and only if $\{x_1-x_0,  \dots, x_k-x_0 \}$ is linearly independent (see, e.g., \cite[Theorem 3.1.3]{Van Mill}).

Observe that $A\subset H$ is a $k$-dimensional flat if and only if there exists an affinely independent set $\{x_0,  \dots, x_k \}\subset A$ such that  

$$A=\Aff(\{x_0, \dots, x_k\}).$$

If $\{a_0,\dots, a_k\}$ is an affinely independent set, every element $x$ in the convex hull $\cco\{a_0,\dots ,a_k\}$ can be uniquely written as a convex sum of the points $a_0,\dots ,a_k$. This allows us to divide the simplex $\cco\{a_0,\dots ,a_k\}$ into two disjoint sets: the relative interior and the relative border of $\cco\{a_0,\dots ,a_k\}$.  Recall that  the \textit{relative interior}  (also call \textit{radial interior}) of $\cco \{a_{0},\dots ,a_k\}$  is defined as the set 

$$\cco^\circ \{a_{0},\dots ,a_k\}:=\left\{\sum_{i=0}^{k}\lambda_i a_i:\sum_{i=0}^{k}\lambda_i =1,\;  \lambda_i\in (0,1)\;\forall \;i=0,\dots, k\right\},$$
while the \textit{relative border} is the set $\cco \{a_{0},\dots ,a_k\}\setminus \cco^\circ \{a_{0},\dots ,a_k\}$.

The following definition is the affine version of a so-called linearly independent family of open sets (\cite[Page 43]{Resende-Santos}).

\begin{definition}
Let $U_1,\dots, U_k$ be a finite collection of open sets in $H$. We say that
$\{U_1,\dots, U_k\}$  is an affinely independent family of open sets (or an affinely independent family, for simplicity) if for every choice of points $u_i\in U_i$, the set $\{u_1,  \dots, u_k\}$ is an affinely independent set. 
\end{definition}

In the previous definition, if we replace the word ``affinely'' by ``linearly'', we obtain the definition of a \textit{linearly independent family of open sets}.

\begin{lemma}(cf. \cite[Lema 7.1]{Resende-Santos})\label{l:fam afin indep}
Let $\{a_0,\dots, a_k\}\subset H$ be an affinely independent set. Then there exists a positive number $\delta$ such that the collection $\{B(a_i,\delta)\}_{i=0}^k$ is an affinely independent family. 
\end{lemma}

\begin{proof}
Let us assume that such $\delta$ does not exist. Thus, for each $i=1,\dots, k$ we can find a sequence $(x_i(m))_{m\in\mathbb N}\subset H$ such that 
\begin{equation}\label{eq. Condicion x_i}
    x_i(m)\in B\left(a_i, \frac{1}{m}\right )
\end{equation}
and the set $\{x_0(m),\dots, x_k(m)\}$ is not affinely independent for any $m\in\mathbb N$. This means that we can find, for each $i=0,\dots, k$, a real sequence $(\lambda_i(m))_{m\in\mathbb N}$ with the following properties 
\begin{itemize}
\item[(F1)] $\sum_{i=0}^k\lambda_i(m)x_i(m)=0$,
    \item[(F2)] $\sum_{i=0}^k\lambda_i(m)=0$,
    \item[(F3)] $\lambda (m):=(\lambda_0 (m), \lambda_1(m),\dots, \lambda_k(m))\in \mathbb R^{k+1}\setminus \{0\}$.
\end{itemize}

Notice that the choice of each $x_i(m)$ (for $i=1,\dots, k$ and $m\in\mathbb N$) implies that $$\lim\limits_{m\to\infty} (x_0(m), \dots, x_k(m))=(a_0,\dots, a_k)\in H^{k+1}.$$
Furthermore, 
property-(F3)  guarantees that $\Vert \lambda(m) \Vert>0$ and therefore 
$$\alpha(m):=\big(\alpha_0(m),\dots, \alpha_k(m)\big)\in\mathbb S^{k}\subset\mathbb R^{k+1},$$
where 
$$\alpha_i(m)=\frac{\lambda_i(m)}{\Vert \lambda(m) \Vert},$$
for each $i=0,\dots, k$ and $m\in\mathbb N$. Also, by property-(F2), each term of the sequence $\big(\alpha(m)\big)_{m\in\mathbb N}$ satisfies that
\begin{equation}\label{eq:suma cero}
    \sum_{i=0}^k \alpha_i(m)=0. 
\end{equation}
Since $\mathbb{S}^{k}$ is compact, there exists a subsequence of $\big(\alpha(m)\big)_{m\in\mathbb N}$ that converges to a certain element $y=(y_0,\dots, y_k)\in\mathbb S^k$.
Using the fact that the map $\mathbb R^{k+1}\to\mathbb R$ given  by 
$$(x_0,\dots,x_k)\to\sum_{i=0}^k x_i$$ is continuous, we infer from equality~(\ref{eq:suma cero})  that $\sum_{i=0}^ky_i=0$.
Finally, since the map $\mathbb R^{k+1}\to H$ defined as 
$$(x_0,\dots,x_k)\to\sum_{i=0}^k x_ia_i$$ is continuous, we conclude from property (F1) that $\sum_{i=0}^k y_i a_i=0.$ However, this is a contradiction since $\{a_0, \dots, a_k\}$ is affinely independent and $(y_0,\dots, y_k)\in \mathbb S^{k}$ is not the zero vector. 
\end{proof}

\subsection{Q-manifolds}

Recall that the \textit{Hilbert cube} is the topological product $$Q:=\prod_{i=1}^{\infty}[-1,1].$$  A \textit{Hilbert cube manifold} ( or simply a $Q$-\textit{manifold}) is a separable metrizable space that admits an open cover made of sets homeomorphic to open subsets of $Q.$ 

The following remark can be easily deduced from the definition of a $Q$-manifold.

\begin{remark}\label{r: abiertos en Q manifolds}
    Let $M$ be a separable metrizable space. 
    \begin{enumerate}
        \item If $M$ is a $Q$-manifold, then every nonempty open subset of $M$ is a $Q$-manifold. 
        \item If every point $x\in M$ has an open neighborhood homeomorphic with a $Q$-manifold, then $M$ is a $Q$-manifold.        
     \end{enumerate}
\end{remark}

Let us recall the following result from the theory of $Q$-manifolds that will be used later in the paper. 

\begin{theorem}\label{thm: X times Q}(\cite[Theorem 44.1]{Chapmanbook})
If $X$ is a $\mathrm{ANR}$ then $X\times Q$
is a $Q$-manifold. 
\end{theorem}

In \cite[Main Theorem]{Sakai-Yang-2007} it was proved that the hyperspace $(\mathcal K^n,\tau_F)$ is a $Q$-manifold homeomorphic to $\mathbb R^{n}\times Q$.
In particular, this implies that every subspace of $\mathcal K^n$ is separable.

As far as we know, the topology structure of the hyperspace $\mathcal K^n_n$  (equipped with the Fell topology) is unknown. However, we have the following.

\begin{proposition}\label{p: topologia de Knn}
 For any $n\geq 2$, the following holds
    \begin{enumerate}[\rm(1)]
        \item $(\mathcal K^n_n, \tau_{LV})$ is open in $(\mathcal K^n, \tau_{LV})$.
        \item $(\mathcal K^n_n, \tau_{F})$ is open in $(\mathcal K^n, \tau_{F})$.
        \item $(\mathcal K^n_n, \tau_{F})$ is a $Q$-manifold.
    \end{enumerate}  
\end{proposition}

\begin{proof}
(1) Let $A\in \mathcal K_n^n$ and pick $a_0, \dots, a_n\in A$ and $\delta>0$ such that the family $\{B(a_i, \delta) \}_{i=0}^n$ is affinely independent. 
Then, if $B\in\mathcal K^n$ is such that $B\cap B(a_i, \delta)\neq \emptyset $ for all $i=0, \dots, n$, we can find points $b_i\in B\cap B(a_i, \delta)$ such that $\{b_0,\dots ,b_n\}$ is affinely independent. Therefore, the simplex $\sigma:=\cco(\{b_0, \dots ,b_n\})$ has nonempty interior in $\mathbb R^{n}$ and since $B$ is convex, we conclude that $\sigma\subset B$. Thus, $B$ has nonempty interior and therefore $B\in\mathcal K^{n}_n$. From this we conclude that
$(\mathcal K^n_n, \tau_{LV})$ is open in $(\mathcal K^n, \tau_{LV})$.

(2) Since the lower Vietoris topology is weaker than the Fell topology, this follows directly from (1).

(3) By \cite[Main Theorem]{Sakai-Yang-2007}, $(\mathcal K^n, \tau_F)$ is a $Q$-manifold homeomorphic to $\mathbb R^{n}\times Q$. Therefore, by Remark~\ref{r: abiertos en Q manifolds}, every nonempty open subset of $\mathcal K^n$ is also a $Q$-manifold. In particular
$(\mathcal K^n_n, \tau_{F})$ is a $Q$-manifold.
\end{proof}

On the other hand, the topology of $\big(\mathcal K^n_{n,b},d_H\big)$ was described
in \cite[Corollary 3.11]{AntonyanNatalia}).
\begin{theorem}\label{th:cb(n)}\cite[Corollary 3.11]{AntonyanNatalia}
  If $n\geq 2$, the hyperspace $\big(\mathcal K^n_{n,b},d_H\big)$ is a $Q$-manifold homeomorphic to $Q\times\mathbb{R}^{\frac{n(n+3)}{2}}$.
\end{theorem}

Notice that $\mathcal{K}_{n,b}^n$ is precisely the family of all compact bodies of $\mathbb R^n$. Namely,
$$\mathcal{K}_{n,b}^n=\{A\in\mathcal K^n: \Int (A)\neq\emptyset
\text{ and }A\text{ is compact}\}.$$

On the other hand, notice that each element in $\mathcal{K}_0^H$ is a singleton $\{x\}$ with $x\in H$.  Hence, the following fact can easily be proved.

\begin{fact}\label{f: H homeomorfo a los singuletes}
    The hyperspaces $(\mathcal K_0^H, \tau_{LV})$ and $(\mathcal K_0^H, \tau_{F})$ are homemorphic with $H$.  
\end{fact}

\subsection{Metric projections on Hilbert spaces}\label{Subsection metric projections}

For every $A\in \mathcal K^H$ and $x\in H$, let $\pi_A(x)$ denote the point of $A$ closest to $x$. 
In this case $\pi_A(x)$  is called the \textit{metric projection of $x$ on $A$} and the associate map $\pi_A:H\to A$ is called the \textit{metric projection of $H$ onto $A$}.
Clearly $\pi_A(x)=x$ if and only if $x\in A$.
Furthermore, the map $\pi_A$ satisfies the following useful properties:

\begin{itemize}
    \item[(P1)] If $x\notin A$, then $\pi_A(x)\in \partial A$.
    \item [(P2)] For all $x\in H$ and $a\in A$, the following inequality holds
$$\langle x-\pi_A(x), a-\pi_A(x) \rangle
\leq 0.$$

    \item [(P3)] The map $\pi_A$ is non expansive, i.e., 
    $$\|\pi_A(x)-\pi_A(y)\|\leq \|x-y\|,\; \text{for all } x,y\in H.$$

    \item [(P4)] If $V$ is a closed vector subspace of $H$, then for every $x\in H$, the vector $x- \pi_V(x)$ is orthogonal to $V$. That is, $\langle x-\pi_V(x), v\rangle =0$ for every $v\in V$. In this case, $\pi_V$ is a linear map called the \textit{orthogonal projection} of $H$ onto $V$ and satisfies the following equality 
    $$\pi_V(x)+\pi_{V^\perp}(x)=x\quad \text{for every } x\in X,$$
where $V^\perp$ denotes the orthogonal complement of $V$. 

\end{itemize}

Let $F=V+a$ be a closed flat where $V$ is a closed linear subspace of $H$ and $a\in H$. Then, for every $x\in H$,  $v\in V$ and $w\in F$, the following equalities hold:
\begin{itemize}
 \item[(P5)] $\pi_F(x)=\pi_V(x-a) +a$.
 \item [(P6)] $\langle x-\pi_F(x),v\rangle=0$.
 \item [(P7)] $\langle x-\pi_F(x),w-\pi_F(x)\rangle=0$.
\end{itemize}

We refer the reader to \cite[Lema 6.54]{Aliprantis} for the proof of properties (P1)-(P4). Properties (P5)-(P7) can easily be deduced from the previous ones.

\begin{remark}\label{rem:proyeccion unico ortogonal}
If $F$ is a closed flat and $x\in H$ is fixed, the point $\pi_F(x)$ is the only point in $F$ that satisfies property (P7). Indeed, if $w_0\in F$ is a point such that $\langle x-w_0,w-w_0\rangle=0$ for every $w\in F$, then 
\begin{align*}
   \|\pi_F(x)-w_0\|^2&=\langle \pi_F(x)-w_0,\pi_F(x)-w_0\rangle\\
   &= \langle \pi_F(x)-x+x-w_0,\pi_F(x)-w_0\rangle\\
&=\langle \pi_F(x)-x,\pi_F(x)-w_0\rangle+ \langle x-w_0,\pi_F(x)-w_0\rangle\\
&=\langle x-\pi_F(x),w_0-\pi_F(x)\rangle+ \langle x-w_0,\pi_F(x)-w_0\rangle =0,
\end{align*}
    since both $\pi_F(x)$ and $w_0$ are elements of $F$.
\end{remark}

Recall that a \textit{hyperplane} in $H$ is a closed flat $F\subset H$ of codimension 1. In particular, given a vector $a\in H\setminus\{0\}$, the hyperplane through $a$ orthogonal to $a$ is the set
$$H_a:=\{x\in H:\langle x-a,a\rangle=0\}.$$

In this case, it is not difficult to verify that for every $x\in H$, the point $w_0$ defined as
$$w_0:=x+a-\frac{\langle x,a\rangle}{\|a\|^2}a$$
belongs to $H_a$. Also, a simple calculation shows that $\langle x-w_0,w-w_0\rangle=0$ for every $w\in H_a$. Hence, by  Remark~\ref{rem:proyeccion unico ortogonal}, we conclude 
\begin{itemize}
    \item[(P8)] For every $a\in H\setminus \{0\}$ and $x\in H$, $\pi_{H_a}(x)=x+a-\frac{\langle x,a\rangle}{\|a\|^2}a$.
\end{itemize}

For every $A\in \mathcal K^H$, let us denote by $p(A)$ the nearest point of $A$ to the origin $0\in H$.  Namely, 
\begin{equation}\label{eq:p}
    p(A):=\pi_A(0),
\end{equation}
where $\pi$ is the metric projection onto $A$ defined above.

\begin{lemma}\label{l:continuidad funcion p}
\begin{enumerate}[\rm(1)]
    \item The map $\nu:(\mathcal K^H,d_{AW})\to \mathbb R$ given by $\nu(A)=\|p(A)\|$ is continuous.  
    \item The map $p:(\mathcal K^H,d_{AW})\to (H, \|\cdot\|)$  is continuous. 
    \item In particular, $p:(\mathcal K^n,\tau_F)\to (\mathbb R^n, \|\cdot\|)$ is continuous.
\end{enumerate}
\end{lemma}

\begin{proof}
(1) Let $A\in\mathcal K^H$ and $\varepsilon>0$. Pick $N\in\mathbb N$ large enough so that $1/N<\varepsilon$ and $\|a\|+1<N$, where $a:=p(A)$.
If $B\in\mathcal K^H$ satisfies $d_{AW}(A,B)<1/N$, by Lemma~\ref{lem:Lemma-Ananda} it follows that
\begin{equation}\label{eq:1}
    \sup_{\|x\|<N}|d(x,A)-d(x,B)|<1/N.
\end{equation}

In particular, since $a\in B(0, N)\cap A$, we have $d(a, B)<1/N$. Thus, there exists $\beta\in B$ such that $d(a, \beta)<1/N$. Let $b:=p(B)\in B$. Hence
\begin{equation}\label{des:1}
    \|b\|\leq \|\beta\| \leq \|\beta-a\|+\|a\|<1/N+\|a\|<\varepsilon+\|a\|.
\end{equation}

On the other hand, since $\|b\|<1/N+\|a\|\leq 1+\|a\|$, by the choice of $N$ we also have $b\in B(0,N)\cap B$.
Using inequality (\ref{eq:1}) we obtain that $d(b, A)<1/N$. Thus, there exists $\alpha\in A$ such that $\|b-\alpha\|<1/N$ and therefore
$$ \|a\| \leq \|\alpha\|\leq \|\alpha-b\|+\|b\|<1/N+\|b\|<\varepsilon+\|b\|.$$
This inequality,  in combination with inequality (\ref{des:1}), completes the proof of (1).

(2) Again, let $A\in\mathcal K^H$ and $\varepsilon>0$. As in the proof of (1), denote $a:=p(A)$. Let  $M\in\mathbb N$ be big enough such that $M>\max\{\|a\|+1, 1/\varepsilon\}$ and 
$$\frac{4\|a\|}{M}+\frac{1}{M^2}<\varepsilon^2.$$
By the first part of this lemma, we can find $\delta>0$ such that $|\nu(B)-\nu(A)|<1/M$ for every $B\in\mathcal K^H$ with $d_{AW}(A, B)<\delta$.

Let $\eta=\min\{\delta, 1/M\}$ and consider $B\in\mathcal K^H$ with $d_{AW}(A, B)<\eta$. Denote $b:=p(B)$. If $a=0$, then $\|b-a\|=\|b\|=|\nu(B)-\nu(A)|<1/M<\varepsilon$. In this case $p$ would be continuous at $A$, as desired. Thus, to prove the general case, we can assume that $a\neq 0$. 

Notice that inequality $|\nu(B)-\nu(A)|<1/M$ implies that 
\begin{equation}\label{ineq:importante}
\|b\|=\nu(B)<\nu (A)+1/M=\|a\|+1/M\leq \|a\|+1< M 
\end{equation}

Hence, $b\in B(0,M)$ and therefore, by Lemma~\ref{lem:Lemma-Ananda}, there exists a point $y\in A$ such that $\|y-b\|<1/M$.

\textbf{Claim.} $\langle b,a\rangle \geq \|a\|^2-\frac{\|a\|}{M}$. Indeed, by property (P2),
\begin{equation*}
   A\subseteq \{x\in H: \langle x-a,a\rangle \geq 0\}=\{x\in H:\langle x,a\rangle\geq \|a\|^2\}=:C.
\end{equation*}
If $b\in C$, we are done. Otherwise, by (P1), the point $\pi_C(b)\in\partial C$. Notice that $\partial C$ is precisely the hiperplane through $a$ orthogonal to $a$. Namely $\partial C=H_a=\{x\in H: \langle x-a,a\rangle = 0\}$. Thus, by property (P8), 
\begin{equation}\label{eq:distancia a C}
  d(b, C)=\left\|b-\left(b+a-\frac{\langle b,a\rangle}{\|a\|^2}a\right)\right\| 
= \left |\|a\|-\frac{\langle b,a\rangle}{\|a\|}\right|
\end{equation}
Since $y\in A\subset C$, we also have that
\begin{equation}\label{eq:distancia menor que 1/M}
    d(b, C)\leq d(b,A)\leq \|b-y\|<\frac{1}{M}. 
\end{equation}
Finally, combining equations (\ref{eq:distancia a C}) and (\ref{eq:distancia menor que 1/M}),
we obtain
$$\|a\|-\frac{\langle b,a\rangle}{\|a\|}\leq \left |\|a\|-\frac{\langle b,a\rangle}{\|a\|}\right|<\frac{1}{M},$$
which directly implies $\langle b,a\rangle \geq \left(\|a\|^2-\frac{\|a\|}{M}\right)$, as desired. 
\QEDA

Now, by inequality (\ref{ineq:importante}), $\|b\|<\|a\|+1/M$. Finally, we can combine this last inequality with the claim and the choice of $M$, to obtain that
\begin{align*}
    \|b-a\|^2&=\|b\|^2+\|a\|^2-2\langle b,a\rangle\\
    &<\left(\|a\|+\frac{1}{M} \right)^2+\|a\|^2-2\langle a,b\rangle\\
    &\leq \left(\|a\|+\frac{1}{M} \right)^2 +\|a\|^2-2\left(\|a\|^2-\frac{\|a\|}{M}\right)\\
    &=\frac{4\|a\|}{M}+\frac{1}{M^2}<\varepsilon^2.
\end{align*}

 This completes the proof of the lemma.

\end{proof}

\section{The topology of the Grassmann manifolds}\label{s:Grassmann}

Let us denote by $\Gr_k(H)$ the set of all $k$-dimensional linear subspaces of $H$. The family $\Gr_k(H)$ is called the  \textit{Grassmann manifold}  (also called the \textit{grassmannian}) 
 \textit{of $k$-dimensional subspaces of $H$}.
On the other hand, we will denote by $\ag_k(H)$ the family of all $k$-dimensional flats in $H$. $\ag_k(H)$ is called the \textit{affine grassmanian of $k$-dimensional flats of $H$}.
 When $H=\mathbb R^n$, we simply write $\Gr_k(n)$ or $\ag_k(n)$, accordingly.

There are several (equivalent) ways to define a topology in $\Gr_k(H)$ (and $\ag_k(H)$) as a quotient space. 
For $\Gr_k(H)$,  one of the most common constructions is the next one. Consider the space 
$$L(k,H):=\{T:\mathbb R^k\to H: T \text{ is linear and injective}\}$$
equipped with the point-open topology.

Thus, the map $p_k:L(k,H)\to \Gr_k(H)$ defined by $p_k(T)=\im T:=T(\mathbb R^k)$ is a surjection and therefore we can endow $\Gr_k(H)$ with the quotient topology induced by $p_K$
\footnote{For the complex case, we can replace $\mathbb R^k$ by $\mathbb C^k$ and the same construction holds}. We refer the reader to \cite{Alexandrino Bettiol}, \cite{Baker}, \cite{Milnor Stasheff} and \cite{Steenrod} for more information and other equivalent quotient topologies on $\Gr_k(H)$.

\subsection{The grassmannians as hyperspaces} Observe that both $\Gr_k(H)$ and $\ag_k(H)$
are subsets of $\mathcal K^H$. Hence, every topology on $\mathcal K^H$ induces a topology on $\Gr_k(H)$ and $\ag_k(H)$. In \cite{Resende-Santos}, the authors proved that the quotient topology on $\Gr_k(H)$ induced by $p_K$ coincides with the Lower Vietoris Topology. It is worth mentioning that even if their proof was made for the complex case, the same ideas work for the real case.
Even more, we can use the same ideas to prove that the lower Vietoris topology in $\ag_k(H)$ is equivalent to a natural quotient topology on $\ag_k(H)$.
For the sake of completeness, we include a complete proof of these facts.

Consider the set
$\Aff_k(H)$ consisting of all affine injective transformations $T:\mathbb R^k\to H$. Namely, 
 $\varphi \in \Aff_k(H)$ iff  $\varphi=T_u+\sigma$ where $T_u: \mathbb R^k \to H$ is the constant map $T_u(x):=u$, and $\sigma:\mathbb R^k \to H$ is a linear injective map. 
The set $\Aff_k(H)$ endowed with the point-open topology generated by the sets $N(x,W):=\{T\in \Aff_k(H): T(x)\in W \}$, where $x\in\mathbb R^k$ and $W\subset H$ is open, becomes a topological space.

Define a map $q_k:\Aff_k(\mathbb R^n) \to \ag_k(H)$ by $q_k(\varphi):=\im(\varphi)$. Clearly $q_k$ is a surjection and therefore it induces a quotient topology, $\tau_{q_k}$, on $\ag_k(H)$.

\begin{theorem}\label{t:topologia cociente y lower coincide}
In $\ag_k(H)$, the quotient topology $\tau_{q_k}$ and the lower Vietoris topology are the same.  
\end{theorem}

\begin{proof}
First,  consider an open set
$U \subset H$ and let us show that $U^-$ belongs to $\tau_{q_k}$. Namely, we shall prove that
$$q_k^{-1}(U^-)=\{\varphi \in  \Aff_k(H) : \im(\varphi) \cap U \neq \emptyset \},$$
is open in  $\Aff_k(H)$. Indeed, if $\varphi \in q_{k}^{-1}(U^-)$, then there exists a point $x \in \mathbb{R}^k$ such that $\varphi(x) \in U$. This directly implies that $\varphi \in N(x,U) \subset p^{-1}(U^-)$, and therefore $\tau_{LV}\subset \tau_{q_k}$.

For the other inclusion, consider $\mathcal W\in \tau_{q_k}$ and let us prove that $\mathcal W$ belongs to $\tau_{LV}$.

 If $A \in \mathcal{W}$, then there exists an affine transformation  $\psi \in q_{k}^{-1}(\mathcal{W})$  such that $q_k(\psi)=A$. Since $\mathcal W$ belongs to $\tau_{q_k}$, there exist points  $x_1,\dots, x_r \in  \mathbb R^k$ and open sets $U_1, \dots, U_r\subset H$, satisfying that 
\begin{equation} \label{eq: vecindad basica en cociente}
\begin{split}
\psi \in \bigcap_{i=1}^r N(x_i,U_i) \subset q_{k}^{-1}(\mathcal{W})
\end{split}
\end{equation}

Without loss of generality, we can assume that 
for a certain $m \leq r$ the set $\{x_1,\ldots, x_m\}$ is affinely independent and $\{x_1,\ldots, x_m,x_j\}$ is affinely dependent for every $x_j\in\{x_1,\dots, x_r\}\setminus\{x_1,\dots ,x_m\}.$
Since $\psi$ is an affine injective map, the set 
$S=\{ \psi(x_1),  \dots, \psi(x_m)\}$ is affinely independent in  $H$ and $\psi(x_i) \in U_i$ for every $i=1,\dots,r$. By Lemma~\ref{l:fam afin indep}, we can find an affinely independent family $U'_1,\dots,U'_m\subset H$ where each $U'_i$ is open in $H$ and
$$\psi(x_i) \in U'_i\subset U_i,$$ for each $i=1,\dots, m$. 

Furthermore, by the choice of $m$, we can find for every $j\in\{1,\dots, r\}$ real scalars $\alpha_{j1}, \dots ,\alpha_{jm}$ with the property that 

$$x_j=\sum_{i=1}^m \alpha_{ji} x_i,\;\text{and }\; \sum_{j=1}^m \alpha_{ji}=1.$$
Clearly, if $j\in \{1,\dots ,m\}$, the scalar $\alpha_{ji}$ coincides with the Kronecker delta $\delta_{ji}$.
Also, since $\psi$ is an affine transformation, notice that
$$\sum_{i=1}^m \alpha_{ji} \psi(x_i)=\psi(x_j)\in U_j.$$
By the continuity of the addition and the scalar multiplication on $H$, we can find open sets $W_1,\dots, W_m\subset H$ such that $\psi(x_i)\in W_i\subset U_i'$ for every $i\in\{1,\dots,m\}$, and 
\begin{equation}\label{eq: contencion W en U}
    \sum_{i=1}^m \alpha_{ji} z_i\in U_j\;\text{ for every }j=1,\dots, r,
\end{equation}
provided that each $z_i\in W_i$ for $i=1,\dots, m$. Notice that

$$A \in \mathcal{U}:=\bigcap_{i=1}^m W_i^-\in \tau_{LV}.$$
To finish the proof, it suffices to verify that $\mathcal U\subset \mathcal W$.
Let $B\in\mathcal U$ and $\phi \in \Aff_k(H)$ be such that $q_k(\phi)=B$. Hence $B\in W_i^-$ for each $i=1,\dots, m$ and therefore we can find points $y_1,\dots, y_m\in\mathbb R^k$ such that $\phi(y_i) \in W_i\subset U_i'$, $i=1,\dots, m$.
 By the choice of $\{U_i'\}_{i=1}^m$ and the fact that $\phi\in \Aff_k (H)$, we conclude that $\{y_1,\dots ,y_m\}$ is an affinely independent set.  Thus, we can find an affine bijection $\sigma:\mathbb R^k\to\mathbb R^k$ such that $\sigma(x_i)=y_i$ for every $i=1,\dots, m$. Notice that in this case $\phi\circ\sigma\in \Aff_k(H)$ and 
 $$B=\im (\phi)=\im (\phi\circ \sigma)=q_k(\phi\circ\sigma).$$

Since $\phi(y_i)\in W_i$ for every $i=1,\dots, m$, we can use formula (\ref{eq: contencion W en U}) to conclude that 
 $$\phi(\sigma(x_j))=\sum_{i=1}^m \alpha_{ji} \phi\left( \sigma \left(   x_i\right) \right)=\sum_{i=1}^m \alpha_{ji} \phi\left(y_i \right) \in B \cap U_j,$$
and therefore
 $\phi \circ \sigma \in N(x_j,U_j)$,  for all $j=1,\dots, r$.
 Finally, we can use formula~(\ref{eq: vecindad basica en cociente})
  to infer that $\phi \circ \sigma\in q_k^{-1}(\mathcal W)$ and then
  $$B=q_k(\phi \circ \sigma)\in   \mathcal W,$$ 
as desired.
 \end{proof}

In addition to the lower Vietoris topology, there are other hyperspace topologies in $\mathcal K^H$ that induce a good topology in $\Gr_k(H)$ and $\ag_k(H)$.  That is the case of the Fell and the Attouch-Wets topology. However, it is interesting to point out that in $\ag_k(H) $ (and in $\Gr_k(H)$ as well), these two topologies coincide with the lower Vietoris topology (and therefore with the quotient topology induced by the maps $p_k$ and $q_k$, respectively).

\begin{theorem}\label{t:lower and Fell coincide}
In $\ag_k(H)$, the lower Vietoris topology and the Fell topology coincide. 
\end{theorem}

\begin{proof}

Since the lower Vietoris topology is always weaker than the Fell topology, we only need to prove that for every compact set $C\subset H$, the sub-basic set $(H\setminus C)^{+}$ belongs to $\tau_{LV}$.
In order to do that, let $A\subset (H\setminus C)$ be a $k$-dimensional flat and pick $\{a_0, \dots ,a_k\}\subset A$ a set of affinely independent points. 

Since $A$ is a flat and $A\cap C=\emptyset$,  we infer that the set $\{c,a_0,\dots, a_k\}$ is affinely independent for every $c\in C$. Thus, we can use Lemma~\ref{l:fam afin indep} to find, for each $c\in C$, a positive value $\delta_c>0$ such that the collection $$\{B(c, \delta_c), B(a_0, \delta_c),  \ldots, B(a_k, \delta_c)\}$$ is affinely independent.

Now, since $C$ is compact, there exist $c_1, c_2, \dots, c_m\in C$ such that
$$C\subset \bigcup_{i=1}^mB(c_i, \delta_{c_i}).$$
 Let $\delta=\min\{\delta_{c_1},\dots, \delta_{c_m}\}$ and  consider the set
$$ \mathcal{U} := \bigcap_{i=0}^{k} B(a_i, \delta)^-=[a_0,\dots a_k; \delta].$$
Clearly $\mathcal{U}$ is an open neighborhood of $A$ with respect to the lower Vietoris topology.
Let us prove that $\mathcal{U}\subset (H\setminus C)^+$. If $B\in \mathcal{U}$ is a $k$-dimensional flat with $B\cap C\neq\emptyset$, we can find a point $b\in B\cap C$ and $j\in\{1,\dots ,m\}$ such that $b\in B(c_j, \delta_{c_j})$.
Furthermore, since $B\in \mathcal{U}$, for every $i=0,\dots, k$ there exists a point $b_i\in B(a_i, \delta)\cap B\subset B(a_i, \delta_{c_j})\cap B$. Since the set $\{b_0,\dots ,b_k\}$ is affinely independent and $B$ is a $k$-dimensional flat, we also have that 
$$b\in B=\Aff(\{b_0,\dots ,b_k\}).$$

This contradicts that $\{B(c_j, \delta_{c_j}), B(a_0, \delta_{c_j}), \dots, B(a_k, \delta_{c_j})\}$ 
is an affinely independent family. 
Hence, we can conclude that $B\cap C=\emptyset$ and therefore $\mathcal U\subset (H\setminus C)^+$. Namely, the Fell topology on $\ag_k(H)$ is weaker than the lower Vietoris topology, and therefore both topologies coincide on $\ag_k(H)$, as desired. 

\end{proof}

\subsection{Metric topologies for the Grassmann manifolds}

A natural way to endow $\Gr_k(H)$ and $\ag_k(H)$ with a metric topology is through the Attouch-Wets metric defined previously (see equation~(\ref{eq:dAW})). Since in $\mathcal K^n$ the Fell topology and the one induced by the Attouch-Wets metric coincide, after Theorem~\ref{t:lower and Fell coincide}, we infer that the lower Vietoris topology and the Attouch-Wets metric topology coincide in $\ag_k(n)$.  Surprisingly, this also happens in the infinite-dimensional case.
To see this, we will need the following technical lemma.

\begin{lemma}\label{l: Bola contenida enc asco convexo}
Let $F\in \ag_k(H)$ and $ \{a_0,\dots ,a_k\}\subset F$ be an affinely independent set. Assume that $F\cap B(a,3M)$ is completely contained in the relative interior of $\cco \{a_0,\dots,a_k\}$, where $a\in F$ and $M>0$. Take $G\in [a_0,\dots, a_k; M]$  and $(b_0,\dots, b_k)\in\prod_{i=0}^k B(a_i, M)\cap G$.   If 
$$b\in B(a,M)\cap \cco\{b_0,\dots,b_k\},$$ then $B(b, M)\cap G\subset \cco^\circ\{b_0,\dots,b_k\}$.
\end{lemma}

\begin{proof}
  Assume that $B(b, M)\cap G\not \subset \cco^\circ\{b_0,\dots,b_k\}$. Since $b\in \cco\{b_0,\dots,b_k\}\cap B(b,M)\neq\emptyset$,  we can find an element $z\in B(b, M)$  in the relative border of $\cco\{b_0,\dots,b_k\}$. Thus, there exists $(\lambda_0,\dots,\lambda_k)\in \prod_{i=0}^k[0,1]$ such that $z=\sum_{i=0}^{k}\lambda_i b_i$ with at least one $\lambda_j=0$ and $\sum_{i=0}^{k}\lambda_i=1$. 

Hence, $\alpha:=\sum_{i=0}^{k}\lambda_i a_i$ lies on the relative border of $\cco \{a_0,\dots ,a_{k}\}$. However, by the triangle inequality and Remark~\ref{r:distancia entre combinaciones convexas}
\begin{align*}
    \|\alpha-a\|&\leq \|\alpha-z\|+\|z-b\|+\|b-a\|\\
    &<M+M+M=3M.
\end{align*}
  By our hypotheses on $M$ and $a$, we conclude that $\alpha$ lies in the relative interior of $\cco \{a_0,\dots a_{k}\}$, a contradiction.
\end{proof}

\begin{theorem}\label{t:lower y AW coincide}
       In $\ag_k(H)$, the lower Vietoris topology and the Attouch-Wets topology coincide. 
\end{theorem}

\begin{proof}
Since the lower Vietoris topology is always weaker than the one generated by the Attouch-Wets metric (see, e.g., \cite[Proposition 3.1.5]{Beer1993}), we only need to prove that the identity map $1_{\ag_k(H)}: (\ag_k(H),\tau_{LV})\to (\ag_k(H),d_{AW})$ is continuous.
Let $F\in \ag_k(H)$ and $\varepsilon >0$. 
Choose a natural number $j\in\mathbb N$ such that $1/j\leq \varepsilon $ 

Denote $a:=p(F)$ (where $p$ is the map defined in equation~(\ref{eq:p})) and choose $L>2j+\|a\|+1/j$. Define $M:=L+\|a\|+1/j$ and pick an affine independent set $\{a_0,\dots,a_k\}\subset F$ such that
$$B(a,3M)\cap F\subset \cco^\circ\{a_0,\dots, a_k\}.$$
Hence, we can find positive scalars $\mu_i\in (0,1)$, $i=0,\dots, k$ satisfying the following equations
$$a=\sum_{i=0}^k\mu_ia_i\quad\text{ and }\quad\sum_{i=0}^k\mu_i=1.$$
Finally, pick $\delta\in (0,1/j)$ small enough such that $\{B(a_i,\delta)\}_{i=0}^k$ is an affinely independent family. 
Hence $\mathcal{O}:=[a_0,\dots,a_k;\delta]$ is an open neighborhood of $F$ with respect to the lower Vietoris topology on $\ag_{k}(F)$.
Let $G\in \mathcal O$. 
To prove the proposition, we will show that $d_{AW}(F,G)<\varepsilon$. To do this, pick elements $b_i\in G\cap B(a_i, \delta)$, $i=0,\dots, k$, and define $b:=\sum_{i=0}^k\mu_ib_i$. Hence $b\in \cco^\circ \{b_0,\dots, b_k\}$ and, by Remark~\ref{r:distancia entre combinaciones convexas},  $\|a-b\|<\delta<M$. These two conditions allow us to use Lemma~\ref{l: Bola contenida enc asco convexo} and conclude that
$$B(b,M)\cap G\subset \cco^\circ\{b_0,\dots,b_k\}.$$

Let $x\in B(0,j)$. Since $L>2j+\|a\|+1/j>2j+\|a\|=2j+d(0,F)$, by Lemma~\ref{l:distancia igual a distancia en interseccion} we know that $d(x,F)=d(x, F\cap B(0,L)).$ Hence, for every $\eta>0$ we can find $z\in F\cap B(0,L)$ such that 
\begin{equation}\label{d:distancia de e a z}
    \|x-z\|<d(x,F\cap B(0,L))+\eta =d(x,F)+\eta.
\end{equation}

Notice that $\|z-a\|\leq \|z\|+\|a\|<L+\|a\|<M$ and therefore $z$ lies in the relative interior of $\cco\{a_0,\dots a_k\}$. Thus, we can find scalars $t_i\in (0,1)$, $i=0,\dots ,k$ with $\sum_{i=0}^kt_i=1$ and $z=\sum_{i=0}^kt_i a_i$. By Remark~\ref{r:distancia entre combinaciones convexas}, the point $z':=\sum_{i=0}^kt_ib_i$ lies in $B(z,\delta)$. Furthermore,  since $z'\in G$ we can use inequality~(\ref{d:distancia de e a z}) to obtain 
\begin{align*}
    d(x,G)\leq \|x-z'\|\leq \|x-z\|+\|z-z'\|<d(x,F)+\eta+\delta.
\end{align*}
Since the last inequality holds for every $\eta>0$, we infer that 
\begin{equation}\label{e: desigualdad valor absoluto 1}
    d(x,G)-d(x,F)\leq \delta.
\end{equation}

On the other hand, since $\|b\|-\|a\|\leq\|b-a\|<\delta$ and $b\in G$, we can deduce that 
\begin{equation}\label{eq:desigualdad punto cercano a G y b}
    d(0,G)=\|p(G)\|\leq\|b\|<\|a\|+\delta
\end{equation}
Thus, $L>2j+\|a\|+1/j>2j+\|a\|+\delta>2j+d(0, G)$, and  by Lemma~\ref{l:distancia igual a distancia en interseccion}, $d(x,G)=d(x, G\cap B(0,L)).$ Then, for an arbitrary $\eta>0$, we can  find a point $y\in G\cap B(0,L)$ satisfying
\begin{equation}\label{eq: desigualdad distancia de x a G}
    d(x,y)<d(x, G\cap B(0,L))+\eta=d(x,G)+\eta.
\end{equation}

Using inequality~(\ref{eq:desigualdad punto cercano a G y b}) again, we infer that
$$\|y-b\|\leq \|y\|+\|b\|<L+\|a\|+\delta<M.$$
This, in combination with Lemma~\ref{l: Bola contenida enc asco convexo}, allows us to conclude that $y\in \cco^\circ\{b_0,\dots, b_k\}$. Hence, we can write $y=\sum_{i=0}^k\lambda_i b_i$, with $\lambda_{i}\in (0,1)$, and $\sum_{i=0}^{k}\lambda_i=1$.
Define $y':=\sum_{i=0}^k\lambda_i a_i$. Then $y'\in F$ and by inequality~(\ref{eq: desigualdad distancia de x a G}) and Remark~\ref{r:distancia entre combinaciones convexas},  we also have that
\begin{align*}
    d(x,F)&\leq d(x, y')\leq d(x,y)+d(y,y')\\
    &<d(x,G)+\eta+\delta.
\end{align*}
Since the last inequality holds for every positive $\eta$, we conclude that $d(x,F)-d(x,G)\leq \delta$.
This, in combination with inequality~(\ref{e: desigualdad valor absoluto 1}), implies that
$$|d(x,F)-d(x,G)|\leq \delta\quad \text{ for every }x\in B(0,j).$$
From here we infer that  
$$\sup_{\|x\|<j}|d(x,F)-d(x,G)|\leq \delta<1/j,$$
which together with Lemma~\ref{lem:Lemma-Ananda} implies that $$d_{AW}(F,G)<1/j\leq \varepsilon,$$ as desired. 
\end{proof}

In the particular case of $\Gr_k(H)$, there is another useful way to endow this space with a metric topology. That is,  through the Hausdorff distance between the closed unit balls of every pair of elements in $\Gr_k(H)$. More precisely, 
given $V, W$ two linear subspaces of $H$, we define
\begin{equation}\label{eq:gap}
\theta(V,W):=d_H(\mathbb B_V, \mathbb B_W),  
\end{equation}
 where $\mathbb B_V$ and $\mathbb B_W$ stand for the unit closed ball of $V$ and $W$, respectively.
The map $\theta$  defines a metric in the hyperspace of all closed subspaces of $H$ and $\theta(V, W)$ is sometimes called the \textit{gap} between $V$ and $W$. This gap, as well as some other variations of the formula~(\ref{eq:gap}), have been strongly studied by several authors (see, e.g., \cite{Berkson,G-M-Gap, Kadets, Ostrovskii}). The restriction of $\theta$ to $\Gr_k(H)$ has also been used as an easy way to equip $\Gr_k(H)$ with a metric topology (e.g. \cite{Barov}). 

It should be noted that the topology on $\Gr_k(H)$ induced by the gap $\theta$ coincides with the one given by the Attouch-Wets metric. Since we did not find in the literature a proof for this fact, we found it interesting to provide one. 

\begin{proposition}\label{p:gap y attouch-Wets}
    In $\Gr_k(H)$, the topology generated by the gap $\theta$ and the topology generated by the Attouch-Wets metric are the same. 
\end{proposition}

\begin{proof}
Since the origin of $H$ belongs to every element of $\Gr_k(H)$, we can use Lemma~\ref{l:d_H-dAW} to infer that 
\begin{align*}
\theta(V,W)&=\sup\limits_{\|x\|<1}\{|d(x,V)-d(x,W)|\}\leq \sup\limits_{\|x\|<r}\{|d(x,V)-d(x,W)|\}\\
&=d_H(V\cap r\mathbb B, W\cap r\mathbb B),
\end{align*}
for every $r\geq 1$ and $V, W\in \Gr_k(H).$
Now, if $0<\varepsilon\leq 1$, we can find $j\in\mathbb N$ such that $\frac{1}{j+1}<\varepsilon\leq\frac{1}{j}$. Hence, if $d_{AW}(V,W)<\varepsilon$, by Lemma~\ref{l:d_H-dAW}-(5)
$$\theta(V,W)\leq d_H(V\cap j\mathbb B, W\cap j\mathbb B)<\varepsilon.$$

This proves that the identity map $1_{\Gr_k(H)}:(\Gr_k(H),d_{AW})\to (\Gr_k(H), \theta)$ is continuous.

 On the other hand, given $\varepsilon$ and $j$ as described above, choose $\delta<\frac{\varepsilon}{j}$.
Let $V, W\in \Gr_k(H)$ be such that $\theta (V, W)<\delta$. Then, for every $x\in V\cap j\mathbb B$, 
since $\pi_W$ is non-expansive (Property (P3)), $\|\pi_W(x)\|\leq \|x\|\leq j$ and $\frac{1}{j}x\in V\cap\mathbb B$. Thus 
\begin{align*}
   d(x, W\cap j\mathbb B)&=\|x-\pi_W(x)\|=\left \|\frac{j}{j}x-j\pi_W\left(\frac{1}{j}x\right)\right \|\\
   &=j\left \|\frac{1}{j}x-\pi_W\left(\frac{1}{j}x\right)\right \|\leq jd\left(\frac{1}{j}x, W\cap \mathbb B\right)\\
   &\leq j d_H(V\cap \mathbb B, W\cap \mathbb B)=j \theta (V,W)
\end{align*}

    symmetrically, we can prove that $d(x, V\cap j\mathbb B)\leq j\theta (V,W)$ for every $x\in W\cap j\mathbb B$. This allows us to conclude that
    $$d_H(V\cap j\mathbb B, W\cap j\mathbb B)\leq j\theta (V, W)<j\delta<\varepsilon.$$

    Using Lemma~\ref{l:d_H-dAW}-(5) again, we infer that $d_{AW}(V, W)<\varepsilon$, which proves the continuity of the identity map $1_{\Gr_k(H)}:(\Gr_k(H),\theta)\to (\Gr_k(H), d_{AW})$, as needed. 

\end{proof}

\begin{remark}
    After theorems~\ref{t:topologia cociente y lower coincide},\ref{t:lower and Fell coincide} and \ref{t:lower y AW coincide},  we conclude that in $\ag_{k}(H)$, the quotient topology, the lower Vietoris topology, the Fell topology and the Attouch-Wets topology coincide.  This, in combination with Proposition~\ref{p:gap y attouch-Wets}, yields that in $\Gr_k(H)$, the same four topologies coincide with the one induced by the gap (formula~(\ref{eq:gap})).
     Thus, throughout the rest of the paper, we will use these topologies interchangeably,  as needed.
\end{remark}

For any closed linear subspace $W$ of a Hilbert space $H$, let us denote by $W^\perp$ its orthogonal complement.

\begin{remark}
If $V, W$ are two closed linear subspaces of $H$, the gap between $V$ and $W$ can be expressed as 
\begin{equation}\label{eq: formula para ga y operator norm}
    \theta(V, W)=\max\{\|\pi_{V^\perp}\circ\pi_{W}\|, \|\pi_{W^\perp}\circ\pi_{V}\|\}, 
\end{equation}
    where $\|\pi_{V^\perp}\circ\pi_{W}\|$ and $\|\pi_{W^\perp}\circ\pi_{V}\|$ stand for the operator norm of $\pi_{V^\perp}\circ\pi_{W}$ and $\pi_{W^\perp}\circ\pi_{V}$, respectively
\end{remark}

\begin{proof}
Let $V$ and $W$ be two closed linear subspaces of $H$. If $x\in V\cap \mathbb B$, by Property (P3)  $\|\pi_W(x)\|\leq\|x\|\leq 1$ and therefore $\pi_{W\cap \mathbb B}(x)=\pi_{W}(x)$. This, in combination with the fact that $\pi_W(x)+\pi_{W^\perp}(x)=x$, implies that
$$d(x, W\cap B)=\|x-\pi_{W}(x)\|=\|\pi_{W^\perp}(x)\|.$$
Since $V\cap\mathbb B=\pi_{V}(\mathbb B)$, we get
\begin{align*}
    \sup\limits_{x\in V\cap\mathbb B} \{d(x, W\cap \mathbb B)\}&=\sup_{x\in V\cap \mathbb B} \{\|\pi_{W^\perp}(x)\|\}= \sup_{x\in \mathbb B} \{\|\pi_{W^\perp}\circ \pi_V(x)\|\}\\
    &=\|\pi_{W^\perp}\circ \pi_V\|.
\end{align*}

Symmetrically, we can prove that
$$\sup\limits_{x\in W\cap\mathbb B} \{d(x, V\cap \mathbb B)\}=\|\pi_{V^\perp}\circ \pi_W\|,$$
which implies formula~(\ref{eq: formula para ga y operator norm}).
\end{proof}

It is well known that the Grassmann manifolds $\Gr_k(n)$ and $\Gr_{n-k}(n)$ are homeomorphic to each other. This result is usually proved using the fact that
$$\Gr_k(n)\cong O(n)/\big(O(k)\times O(n-k)\big)$$
(see, e.g. \cite[\S 8.4]{Baker}).
However, following the more synthetic aim of this paper, this relation can be easily proved by using the following remark.

\begin{remark} \label{p: orthogonal complent es isometria}
    Let $V$ and $W$ be two closed subspaces of the Hilbert space $H$. Then 
$$\theta(V,W)=\theta(V^\perp,W^\perp).$$
    In particular, the map $(\Gr_k(n), \theta)\to (\Gr_{n-k}(n),\theta)$ given by $W\to W^{\perp}$ is a surjective isometry.
\end{remark}

\begin{proof}
    Since the orthogonal projection is a self-adjoint operator (see e.g. \cite[Proposition 15.47]{Fabian et al}), we clearly have that
    $$\|\pi_{V^\perp}\circ\pi_{W}\|=\|(\pi_{V^\perp}\circ\pi_{W})^*\|=\|\pi_{W}^*\circ\pi_{V^\perp}^*\|=\|\pi_{W}\circ\pi_{V^\perp}\|.$$
    Analogously, $\|\pi_{W^\perp}\circ\pi_{V}\|=\|\pi_{V}\circ\pi_{W^\perp}\|$, which in combination with 
 formula~(\ref{eq: formula para ga y operator norm}) implies the equality $\theta(V,W)=\theta(V^\perp,W^\perp).$
\end{proof}

\section{$\ag_k(n)$ as a fiber bundle over $\Gr_k(n)$}\label{s: fiber bundle of grasmanians}

For every flat $F\in\ag_k(H)$, let $q(F)\in \Gr (H)$ be defined as 
\begin{equation}\label{eq:traslacion de F}
    q(F):=F-p(F)=\{x-p(F): x\in F\},
\end{equation}
where $p(F)$ is the point in $F$ nearest to the origin (equation~(\ref{eq:p})). Namely, $q(F)$ is the subspace of $H$ parallel to $F$.

\begin{lemma}\label{l: q is continuo}
    The map $q:\ag_k(H)\to \Gr_k(H)$ defined in equation (\ref{eq:traslacion de F})
    is a retraction.
\end{lemma}

 \begin{proof} If $W\in \Gr_k(H)$,
 $p(W)=0$ and therefore $q(W)=W-p(W)=W$. Thus, in order to prove that $q$ is a retraction, we only need to prove the continuity of $q$.

Since the Attouch-Wets topology in $\ag_k(H)$ coincides with the lower Vietoris topology, by
Lemma~\ref{l:continuidad funcion p} and Fact~\ref{f: H homeomorfo a los singuletes},  the map $\ag_k(H)\to (\mathcal K^H_0, \tau_{LV})$ given by $F\to \{p(F)\}$ is continuous. This, in combination with Proposition~\ref{p: basic properties}-(1), implies that $q$ is continuous. 
 
 \end{proof}

The topology of $\ag_k(n)$ can be better understood using the fact that $q:\ag_k(n)\to\Gr_k(n) $ defines a fiber bundle structure. Let us recall what this means.

A \textit{fiber bundle} is a tuple $\zeta=(E,B, F, p)$ where 
$E, B$ and $F$, are topological spaces,  $p:E\to B$ is a continuous surjective function (called the \textit{bundle projection}) and  there exists an open covering $\mathcal U$ of $B$
 such that for every $U\in\mathcal U$ there is a homeomorphism $\varphi_U:U\times F\to p^{-1}(U)$ satisfying the following equality:
 $$p(\varphi_U(b,t))=b,\quad\text{for every }(b, t)\in U\times F.$$
 
 Notice that this equality implies that the \textit{fiber over} $b$, $p^{-1}(b)$, is homeomorphic with $F$ for every $b\in B$. In this case $B$ is called the \textit{base space}, $E$ the $\textit{total space}$ and $F$ the \textit{fiber} of the bundle. 

If the fiber $F$ is homeomorphic with a vector space, we obtain the common notion of a \textit{vector bundle}, or an $n$-\textit{plane bundle} in the case $F\cong \mathbb R^{n}$.

We refer the reader to \cite{Spanier} and \cite{Hatcher} (see also \cite{Steenrod}) for more information about fiber bundles.

We did not find in the literature a formal proof that $q:\ag_k(n)\to\Gr_k(n) $ defines a vector bundle projection. Since we will use this fact as well as some constructions made in its proof,  in this section we will provide a complete proof of it.

For every $V\in \Gr_k(H)$, let $\pi_V: H\to V$ be the orthogonal projection onto $V$ (see Subsection~\ref{Subsection metric projections}) and define the following subset of $\Gr_k(H)$:
\begin{equation}\label{eq: definicion abiertos en Gr}
    \widetilde V:=\{W\in \Gr_k(H):\pi_V(W)=V\}\subset \Gr_k(H).
\end{equation}

\begin{theorem}\label{t: W tilde es abierto}
    The family $\{\widetilde V\}_{V\in \Gr_k(H)}$ is an open cover of $\Gr_k(H)$.
\end{theorem}

\begin{proof}
Observe that $V\in\widetilde V$ for every $V\in \Gr_k(H)$. Therefore, $\{\widetilde V\}_{V\in \Gr_k(H)}$ is a cover.  Let $V\in \Gr_k(H)$, and let us prove that $\widetilde V$ is open. 
For that purpose, pick $v_1,\dots,v_k\in V$ linearly independent points. By \cite[Lemma 7.1]{Resende-Santos} there exists $\varepsilon>0$ such that $\{B(v_i,\varepsilon)\}_{i=1}^k$ is a linearly independent collection. If $W\in \widetilde V$, we can use the fact that $\pi_V(W)=V$ to find linearly independent vectors $w_1,\dots,w_k\in W$ such that $\pi_V(w_i)=v_i$. Furthermore, by the continuity of $\pi_V$, there exist $\delta>0$ small enough such that
\begin{enumerate}
    \item $\{B(w_i, \delta)\}_{i=1}^k$ is a linearly independent family and 
    \item $\pi_V\big(B(w_i,\delta)\big)\subset B(v_i, \varepsilon)\cap V$
for every $i=1,\dots, k$.
\end{enumerate}
Now, for every $U\in[w_1,\dots,w_k; \delta]$, let us pick points $u_i\in B(w_i,\delta)$, $1\leq i\leq k$. Thus, the set $\{u_1,\dots, u_k\}$ is linearly independent and $\pi_V(u_i)\in B(v_i, \varepsilon)$ for every $i=1,\dots , k$. By the choice of $\varepsilon$, this implies that $\{\pi_V(u_1),\dots ,\pi_V(u_k)\}$ is a base for $V$ and using the fact that $\pi_V$ is a linear map, we conclude that $\pi_V(U)=V$. Therefore,
$$W\in [w_1,\dots, w_k; \delta]\subset \widetilde V,$$
which proves that $\widetilde V$ is open for every $V\in \Gr_k(H)$.

\end{proof}


\begin{proposition}
    Let $V,W\in \Gr_k(H)$. Then the open sets $\widetilde V$ and $\widetilde W$ defined in equation~(\ref{eq: definicion abiertos en Gr}) are homeomorphic to each other.
\end{proposition}

\begin{proof}
Let $V,W\in \Gr_k(H)$ be two arbitrary $k$-dimensional subspaces of $H$, and let $V^{\perp}$ and $W^{\perp}$ denote their corresponding orthogonal
complements. Since 
$$\text{codim}~V^{\perp}=\text{codim}~W^{\perp}=k<\infty,$$ 
and
$$V\oplus V^{\perp}=H=W\oplus W^{\perp}, $$
we can find a linear isomorphism $T:H\to H$ such that $T(V)=W$ and $T(V^{\perp})=W^{\perp}$. Define a map $\widetilde T:\widetilde V\to\widetilde W$ by
$$\widetilde T(U):=T(U)=\{T(u):u\in U\}.$$
We claim that $\widetilde T$ is a homeomorphism. Indeed,
$T$ is continuous because the induced function is always continuous with respect to the lower Vietoris topology (see, e.g., \cite[Theorem 5.10]{Michael} or \cite[Proposition 3.1-(1)]{DonJuanJonardPerezLopezPoo}). In order to prove that $\widetilde T$ is a homeomorphism, first let us prove that $\widetilde T$ is well defined. Namely, we shall prove that $\widetilde T(U)\in \widetilde W$ for every $U\in \widetilde V$. Assume the opposite is true. Then we can find $U\in\widetilde V$ such that $T(U)\notin \widetilde W$. This implies that $\pi_W\big(T(U)\big)\neq W$ and therefore there exists a non-null vector $y\in T(U)\cap W^\perp$. Pick $u\in U\setminus\{0\}$ such that $T(u)=y\in W^\perp$. Since $T$ is an isomorphism and $T(V^\perp)=W^\perp$, we obtain that
$$u=T^{-1}\big(T(u)\big)=T^{-1}(y)\in T^{-1}(W^\perp)=V^\perp.$$ Hence $\pi_V(u)=0$ and therefore $U$ cannot be in $\widetilde V$, a contradiction. Then, we conclude that $\widetilde T:\widetilde V\to\widetilde W$ is a well-defined continuous map. 
For the same reason, the linear isomorphism $T^{-1}$ induces a continuous (and well-defined) map $\widetilde{T^{-1}}:\widetilde W\to\widetilde V$. Furthermore, it is clear that 
$$\widetilde T\circ\widetilde{T^{-1}}=1_{\widetilde W}\quad\text{ and }\quad \widetilde{T^{-1}}\circ\widetilde T=1_{\widetilde V}.$$
These last equalities prove that $\widetilde T$ is a homeomorphism, as desired. 
\end{proof}



 \begin{lemma}\label{l:w_V es continua}
 Let $W\in\Gr_k(H)$. For any $(V,w)\in\widetilde W\times W$, let us denote by $w_V$ the unique element in $V$ that satisfies $\pi_W(w_V)=w$.
Then the map 
$\Phi_W:\widetilde W\times W\to \bigcup\limits_{V\in\widetilde W} V\subset H$ given by 
\begin{equation}\label{e:w_V}
    \Phi_W(V, w):=w_V
\end{equation} is continuous. 
     \end{lemma}

\begin{proof}
   Let $(V,w)\in \widetilde W\times W$ and $\varepsilon>0$.
Pick $v_0, \dots,v_{k}\in V\cap B(w_V,\varepsilon)$ affinely independent points such that $w_V$ lies in the relative interior of $\cco\{v_0,\dots ,v_k\}$.
   Then we can find $\varepsilon_1>0$ with the property that
   $$V\cap B(w_V,\varepsilon_1)\subset \cco\{v_0,\dots ,v_k\}\subset V\cap B(w_V,\varepsilon).$$

   For each $i=0,\dots,k$, define $w_i:=\pi_W(v_i)$. Since $\pi_W|_{V}$ is a linear isomorphism, we have that
$\{w_0,\dots,w_k\}$ is an affinely independent set and $w=\pi_W(w_V)$ lies in the relative interior of $\cco\{w_0,\dots ,w_k\}$. This implies that there exists $\delta>0$ such that 
   $$B(w, 3\delta)\cap W\subset \cco^\circ\{w_0,\dots ,w_{k}\}.$$
   Now, let $\delta_1\leq \delta$ be small enough such that $\{B(w_i, \delta_1)\}_{i=0}^k$ is an affinely independent family with the property that $$B(w, \delta)\cap W\subset \cco\{w'_0,\dots ,w'_{k}\}$$ for any $(w'_0,\dots,w'_k)\in \prod_{i=0}^{k}\big(B(w_i,\delta_1)\cap W\big) $ (see lemmas~\ref{l:fam afin indep} and \ref{l: Bola contenida enc asco convexo}).

Finally, let $\eta>0$ be small enough such that $\eta\leq \delta_1$, the family $\{B(v_i, \eta)\}_{i=0}^k$ is affinely independent and $B(v_i,\eta)\subset B(w_V,\varepsilon)$ for any $i=0,\dots,k$.
Define $$\mathcal O:=\big([v_0,\dots v_k; \eta]\cap \widetilde W\big)\times \big(W\cap B(w,\delta)).$$ Clearly $\mathcal O$ is an open neighborhood of $(V,w)$ in $\widetilde W\times W$. Furthermore, if $(V',w')\in \mathcal O$,  for every $i=0,\dots, k$ we can find elements $v'_i\in V'\cap B(v_i, \eta)\subset B(w_V, \varepsilon)$. 
Since $\pi_W$ is non expansive (property (P3)), we have that 
$$\|w_i-\pi_W(v_i')\|=\|\pi_W(v_i)-\pi_W(v_i')\|\leq \|v_i-v_i'\|<\eta\leq \delta_1.$$  

Hence, each point $z_i:=\pi_W(v_i')$ lies in $B(w_i, \delta_1)\cap W$ and therefore
$w'\in B(w,\delta)\cap W\subset \cco\{z_0,\dots,z_k\}$.
This implies that we can find scalars $\lambda_i\in [0,1]$, $i=0,\dots, k$ such that 
\begin{align*}
   w'&=\sum\limits_{i=0}^{k}\lambda_i z_i= \sum\limits_{i=0}^{k}\lambda_i\pi_{W}(v_i')\\
   &=\pi_W\bigg(\sum_{i=0}^{k}\lambda_iv_{i}'\bigg).
\end{align*}
   Let $z:=\sum_{i=1}^{k}\lambda_iv_{i}'$. Hence $z\in \cco\{v_0',\dots,v_k'\}\subset B(w_V, \varepsilon)\cap V'$ and $\pi_W(z)=w'$. 
   By the definition of $\Phi_W$ we conclude that $z=w'_{V'}=\Phi_W(V',w')$ and 
   $$\|\Phi_W(V,w)-\Phi_W(V',w')\|=\|w_V-z\|<\varepsilon.$$
   Then the map $\Phi_W$ is continuous, as desired.   
\end{proof}

\begin{theorem}\label{t: grasmaniana afin es haz fibrado sobre la lineal}
Let $q:\ag_k(n)\to\Gr_k(n)$ be the map defined in equation~(\ref{eq:traslacion de F}). \begin{enumerate}[\rm(1)]
    \item For every $W\in\Gr_k(n)$, the set $q^{-1}(\widetilde W)$ is homeomorphic with the topological product $\widetilde W \times \mathbb R^{n-k}$.
    \item The tuple $(\ag_k(n), \Gr_k(n), \mathbb R^{n-k}, q)$ is a vector bundle. 
\end{enumerate}
\end{theorem}

\begin{proof}

We already know that $q:\ag_k(n)\to\Gr_k(n)$ is a continuous surjection (Lemma~\ref{l: q is continuo}) and the family $\{\widetilde W\}_{W\in \Gr_k(n)}$ is an open cover for $\Gr_k(n)$ (Theorem~\ref{t: W tilde es abierto}). Furthermore, for every $W\in \Gr_k(n)$, we clearly have that $W^\perp$ and $\mathbb R^{n-k}$ are linearly isomorphic. Fix $W\in \Gr_k(n)$. Thus, in order to prove the theorem, it is enough to exhibit  a homeomorphism $\varphi_W:\widetilde W\times W^\perp\to q^{-1}(\widetilde W)$ such that 
\begin{equation}\label{eq: lo que debe cumplir varphi}
    q(\varphi _W(V,\omega))=V, \quad \text{ for all }(V,\omega)\in \widetilde W\times W^\perp.
\end{equation}

To that end, define $\varphi_W:\widetilde W\times W^\perp\to q^{-1}(\widetilde W)$ as 
\begin{equation}\label{eq.definicion varphi W}
    \varphi_W(V,\omega):=V+\omega.
\end{equation}
If $(V,\omega)\in \widetilde W\times W^\perp$,
notice that 
\begin{align*}
    q(\varphi _W(V,\omega))&=q(V+\omega)
    =V+\omega-p(V+\omega),
\end{align*}
where $p$ is the map defined in equation~(\ref{eq:p}). By property (P5), $p(V+\omega)=\omega-\pi_V(\omega)$. Using this equality and the fact that $\pi_V(\omega)\in V$, we obtain that
\begin{align*}
    q(\varphi_W(V,\omega))&=q(V+\omega)
    =V+\omega-\omega+\pi_V(\omega)\\
    &=V+\pi_V(\omega)=V.
\end{align*}
This means that  $\varphi_W$ satisfies equation~(\ref{eq: lo que debe cumplir varphi}).
Furthermore, since $V\in \widetilde W$, we also have that $$\varphi_W(V,\omega)\in q^{-1}(V)\subset q^{-1}(\widetilde W).$$
Namely, $\varphi_W$ is well defined. 
On the other hand, the continuity of $\varphi_W$ is guaranteed by Proposition~\ref{p: basic properties}. 

To finish the proof, we still need to prove that $\varphi_W$ is bijective and $\varphi_W^{-1}$ is continuous. 

Let $(V,\omega), (V',\omega')\in \widetilde W\times W^\perp$ be such that $$\varphi_W(V,\omega)=V+\omega=V'+\omega'=\varphi_W(V',\omega').$$
Then $V=q(\varphi_W(V,\omega))=q(\varphi_W(V',\omega'))=V'$.  If $\omega\neq \omega'$, since $\omega'\in V'+\omega'=V+\omega$, there exists a non-null vector $v\in V$ such that $\omega'=v+\omega$. This implies that 
$v=\omega'-\omega \in W^\perp$ and therefore
$\pi_W(v)=\pi_W(\omega'-\omega)=0$. Since $v\neq 0$, we infer that $\pi_W(V)\neq W$, contradicting the fact that $V\in \widetilde W$. Thus $\omega=\omega'$ and therefore the map $\varphi_W$ is injective.

Now, let us prove that $\varphi_W$ is onto. To do this, pick a flat $F\in q^{-1}\big(\widetilde W\big)$ and define $V:=q(F)$. Hence, $F=p(F)+V$ and $V\in \widetilde W$. Therefore we can find an element $v\in V$ such that
\begin{equation}\label{e:onto 1}
    \pi_W(v)=\pi_{W}(p(F))
\end{equation}
Let $w:=p(F)-v$
and observe that $w\in p(F)+V=F$.
On the other hand, since $w=\pi_{W}(w)+\pi_{W^\perp}(w)$, we obtain that
\begin{align*}
   \pi_{W^\perp}(w)&=w- \pi_{W}(w)=w-\pi_{W}(p(F)-v)\\
   &=w-\big(\pi_{W}(p(F))-\pi_W(v)\big)=w-0=w.
\end{align*}
This proves that $w\in W^\perp$, and therefore $(V,w)\in \widetilde W\times W^\perp$. Now simply observe that 
$$\varphi_W(V, w)=V+w=V+p(F)-v=V+p(F)=F,$$
which proves that $\varphi_W$ is a surjective map.
In addition, all previous reasoning allows us to conclude that $\varphi_W^{-1}$ is defined by
$$\varphi_W^{-1}(F):=\big(q(F), p(F)-v \big),$$
where $v$ is the unique vector in $q(F)$ that satisfies equality~(\ref{e:onto 1}). However, by Lemma~\ref{l:w_V es continua}, $v=\Phi_W\big(q(F),\pi_{W}(p(F)) \big)$, where $\Phi_W$ is the continuous map defined in equation~(\ref{e:w_V}). Hence, we can write $\varphi_W^{-1}(F)$ as
\begin{equation}\label{eq: inversa de varphi_W}
    \varphi_W^{-1}(F)=\Big(q(F), p(F)-\Phi_W\big(q(F),\pi_{W}(p(F)) \big)\Big).
\end{equation}

Then the continuity of $\varphi_W^{-1}$ follows from the continuity of $q$, $p$ and $\Phi_W$.

\end{proof}

\section{The hyperspace of $k$-dimensional closed convex sets}\label{s:main}

The aim of this section is to prove our main results. That is, we will prove that $\mathcal K^n_k$ and $\mathcal K^n_{k,b}$ are $Q$-manifolds with a fiber bundle structure over the Grassmann manifold $\Gr_k(n)$.

\begin{proposition}\label{p: casco affin es continua}

The map $\Phi: (\mathcal K^H_k, \tau_{LV}) \to \ag_k(H)$ given by 
\begin{equation}\label{eq: map Phi}
  \Phi(A):=\Aff (A)  
\end{equation}

is a retraction. 
\end{proposition}

\begin{proof}
Since $\Aff(F)=F$ for every $F\in \ag_k(H)\subset \mathcal K^H_k$, we only need to prove the continuity of $\Phi$. To do that, let $U\subset H$ be any open subset. We shall prove that $\Phi^{-1}(U^-)$ is an open subset in the lower Vietoris topology of $\mathcal K^H_k$.

Consider any element $A\in \mathcal K_k^H$ such that $\Phi(A)=\Aff(A)\in  U^-$. Thus, $ \Aff(A)\cap U \neq \emptyset$ and therefore we can find an element $x \in \Aff(A) \cap U$ and $\delta >0$ such that $B(x,\delta) \subset U$. 

By the definition of $\Aff(A)$, there exist affinely independent points   $a_0, \ldots, a_{j} \in A$  (with $j \leq k$) and scalars $\lambda_0, \dots,\lambda_j\in\mathbb R\setminus\{0\}$ such that     $$x= \sum_{i=0}^{j} \lambda_i  a_i \; \text{and} \; \sum_{i=0}^{j} \lambda_i =1. $$ 
Now, define $\varepsilon >0$ as
 $$\varepsilon:= \min \left\lbrace \frac{\delta}{\vert \lambda_i \vert (j+1)} : 0 \leq i \leq j \right\rbrace.$$

Clearly $[a_0,  \dots , a_j;\varepsilon]$ is an open neighborhood of $A$ in the lower Vietoris topology of $\mathcal K^H_k$.  Now, if $B \in [a_0,  \dots, a_j;\varepsilon]$,  we can find points $b_0,  \dots, b_{j} \in B$ with $d(a_i, b_i)< \varepsilon$ for every $i=0,\dots, j$. Therefore, the point $$y:=\sum_{i=0}^{j} \lambda_i b_i $$ belongs to $\Phi(B)=\Aff(B)$ and satisfies 
\begin{align*}
d(x,y)= \Vert x-y \Vert &= \left\Vert \sum_{i=0}^{j} \lambda_i (a_i-b_i) \right\Vert \\
&\leq \sum_{i=0}^{j} \vert \lambda_i \vert \Vert a_i-b_i \Vert < \delta.
\end{align*}
We infer that $y \in B(x, \delta) \subset U $ and hence $\Phi(B)\in  U^-$. This proves that $\Phi$ is continuous as desired. 
\end{proof}

\begin{remark}\label{r:casco Afin es continua}
    Since the Lower Vietoris topology is weaker than the Fell topology, the Hausdorff metric topology, and the Attouch-Wets metric topology, the result in Proposition~\ref{p: casco affin es continua} remains true if we replace $(\mathcal K^H_k, \tau_{LV})$ by $(\mathcal K^H_k, \tau_{F})$, $(\mathcal K^H_{k,b}, \tau_{H})$ or $(\mathcal K^H_k, \tau_{AW})$.
\end{remark}

For any $W\in\Gr_k(n)$, let $$\Phi_W:\widetilde W\times W\to \bigcup\limits_{V\in\widetilde W} V\subset\mathbb R^n$$ be the map defined in Lemma~\ref{l:w_V es continua}.
Consider $\mathcal K_k^W$ the hyperspace of all $k$-dimensional closed convex subsets of $W$ equipped with the Attouch-Wets metric topology.  
Let us define a new map $$\Lambda_W:\widetilde W\times \mathcal K_k^W\to \bigcup_{V\in\widetilde W}\mathcal K_k^V\subset \mathcal K^n_k$$ by means of the following formula:
\begin{equation}\label{f:funcion lambda}
  \Lambda_W(V,A):=\Phi_W(\{V\}\times A)=\{\Phi_W(V,w):w\in A\}.   
\end{equation}

That is, the map $\Lambda_W$ assigns to each pair $(V,A)\in\widetilde W\times \mathcal K_k^W$, the unique set $B\subset V$ such that $\pi_W(B)=A$. In particular, we have that
\begin{equation}\label{eq: proyeccion de lambda  es A}
    \pi_W(\Lambda_W(V,A))=A\quad\text{for every }(V,A)\in\widetilde W\times \mathcal K_k^W.
\end{equation}
Since the restriction $\pi_W|_V$ is a linear isomorphism, the set $\Lambda_W(V,A)$ is a closed $k$-dimensional convex set for every $A\in \mathcal K^W_k$. Hence, the codomain of the map $\Lambda_W$ is well defined. 

On the other hand, notice that if $A\in \mathcal K_{k,b}^W$, then 
\begin{equation}\label{f: la restriccion de lambda funciona}
    \Lambda_W(V, A)\in \mathcal K^V_{k,b}\subset \mathcal K_{k,b}^n.
\end{equation}

Throughout the rest of the section, we will assume that $\mathcal K^n_k$ is equipped with the Attouch-Wets topology (which in this case coincides with the Fell topology) and $\mathcal K^n_{k,b}$ is equipped with the Hausdorff distance topology (which also coincides with the Fell and the Attouch-Wets topologies).
Using these topologies, in the following lemma we will prove that $\Lambda_W$ is continuous.

\begin{lemma}\label{l:w_V inducida es continua}
 Let $W\in\Gr_k(n)$ and consider the map $\Lambda_W$ defined in equation~(\ref{f:funcion lambda}). Then the following holds:
 \begin{enumerate}[\rm(1)]
     \item  $\Lambda_W$  is continuous.
     \item If we equip $\mathcal K_{k,b}^W$ and $\mathcal K_{k,b}^n$ with the topology induced by the Hausdorff distance, then the restriction $\Lambda_{W}|_{\widetilde W\times \mathcal{K}_{k,b}^W}:\widetilde W\times \mathcal{K}_{k,b}^W\to \mathcal{K}_{k,b}^n$ is continuous. 
 \end{enumerate}

 \end{lemma}

\begin{proof}

Recall that in $\mathcal{K}_{k,b}^W$ and $\mathcal{K}_{k,b}^n$, the topology induced by the Hausdorff distance coincides with the one induced by the Atouch-Wets metric. Thus, in order to prove the lemma, it suffices to prove part (1).

Let $\varepsilon >0$ and pick $j\in\mathbb N$ such that $1/j\leq\varepsilon$. Take a positive number $\eta\in (\frac{1}{j+1},\frac{1}{j})$ and fix an element $(V, A)\in \widetilde W\times \mathcal K_k^W$. We shall prove the continuity of the map $\Lambda_W$ at $(V, A)$. 
For that purpose, let $L>2j+\eta+d\big(0,\Lambda_W(V,A)\big)$.
Since $L\mathbb B$ is compact, the set $\pi_W(L\mathbb B)$ is bounded and therefore we can find $M\in \mathbb N$ such that
\begin{equation}\label{f:contenciones proyeccion y bolas}
    \pi_W(L\mathbb B)\subset (M-1)\mathbb B\cap W\subset M\mathbb B\cap W. 
\end{equation}

Observe that $\widetilde W$ is locally compact (since it is an open subset in the compact and metric space $\Gr_k(n)$), then we can find a compact neighborhood $\mathcal O_1\subset \widetilde W$ of $V$.
Hence, the restriction of the continuous map $\Phi_W$ to the compact set $\mathcal O_1\times (M\mathbb B\cap W)$ must be uniformly continuous. In particular, we can find an open neighborhood $\mathcal O\subset \mathcal O_1\subset \widetilde W$ of $V$ and $0<\delta < 1/M$ such that for any $U\in \mathcal O$ and any pair $x,y\in M\mathbb B\cap W $ such that $\|x-y\|<\delta$, the following inequality holds
\begin{equation}\label{d:esigualdad continuiad uniforme}
  \|\Phi_W(V, x)-\Phi_W(U, y)\|<\eta. 
\end{equation}

Let $U\in\mathcal O$ and $B\in \mathcal K_k^W$ be such that $d_{AW}(A, B)<\delta$.
To prove the lemma, we will show that $d_{AW}\big(\Lambda_W(V,A),\Lambda_W(U,B)\big)<\varepsilon$.


Since $\delta < 1/M$, by Lemma~\ref{lem:Lemma-Ananda}, we have that
\begin{equation}\label{eq:diferencia de desitancias menor que delta}
    \sup_{\|x\|<M}\{|d(x,A)-d(x,B)|\}<\delta.
\end{equation}

On the other hand, the choice of $L$ in combination with Lemma~\ref{l:distancia igual a distancia en interseccion}, implies that 
\begin{equation}\label{eq: distancia de lambda en A al 0}
d\big(x, \Lambda_W(V,A)\big)=d\big(x,\Lambda_W(V,A)\cap B(0, L)\big), 
\end{equation}
for every $x\in\mathbb R^n$ with $\|x\|<j$.
\medskip

\textbf{Claim 1.} $d\big(0, \Lambda_W(U,B)\big)\leq d\big(0,\Lambda_W(V,A) \big)+\eta$.

\textit{Proof of Claim 1.}
  Let $z:=p\big(\Lambda_W(V,A)\big)$ be the nearest point of $\Lambda_W(V,A)$ to the origin.  Since $L> d\big(0,\Lambda_W(V,A) \big)$, the point $z$ lies in $B(0,L)$. Then $\|z\|=d\big(0,\Lambda_W(V,A) \big)$ and  $$w:=\pi_W(z)\in A\cap \pi_W(B(0,L))\subset (M-1)\mathbb B\subset B(0,M).$$ 

    By inequality (\ref{eq:diferencia de desitancias menor que delta}), we conclude that
    $$d(w, B)=|d(w,B)-d(w,A)|<\delta.$$

    Hence, we can find a point $b\in B\subset W$ with $\|b-w\|<\delta$. Since $w\in \pi_W(B(0,L))\subset (M-1)\mathbb B$, we conclude that 
    $$\|b\|\leq \|b-w\|+\|w\|<\delta+M-1< 1/M+M-1\leq M$$
Thus $b\in M\mathbb B\cap W$, and since $\|b-w\|<\delta$, we obtain that
$$ \|\Phi_W(V, w)-\Phi_W(U, b)\|<\eta.$$
We can then use the triangle inequality and the fact that $z=\Phi_W(V,w)$ to deduce 
\begin{align*}
    d(0, \Lambda_W(U, B))&\leq\|\Phi_W(U, b)\|\leq \|\Phi_W(V, w)\|+\eta\\
    &= d\big(0,\Lambda_W(V,A)\big)+\eta
\end{align*}
and therefore the claim has been proved. 
$\blacksquare$

From Claim 1 and the choice of $L$, we infer the following inequality:
$$d(0, \Lambda_W(U,B)\big)+2j\leq d\big(0,\Lambda_W(V,A)\big)+\eta+2j<L.$$
This last inequality in combination with Lemma~\ref{l:distancia igual a distancia en interseccion} implies the following claim.

\textbf{Claim 2.}  If $x\in B(0, j)$, then $d\big(x, \Lambda_W(U,B)\big)=d\big(x, \Lambda_W(U,B)\cap B(0,L)\big)$.

Fix $x\in \mathbb R^n$ with $\|x\|<j$. For any $y\in \Lambda_W(U,B)\cap B(0,L)$, we have $b:=\pi_W(y)\in B\cap \pi_W(B(0,L))\subset (M-1)\mathbb B$.
Then, by inequality~(\ref{eq:diferencia de desitancias menor que delta})
$$d(b,A)=|d(b,A)-d(b,B)|<\delta.$$
Therefore we can find a point $a\in A$, with $\|a-b\|<\delta$.
 Since $\|a\|\leq \|a-b\|+\|b\|<\delta+M-1<M$, we get that $a\in M\mathbb B$.
Then, by the choice of $\delta$, we conclude that
$$\|\Phi_W(U,b)-\Phi_W(V,a)\|<\eta.$$
However, since $b=\pi_W(y)$ and $y\in \Lambda_W(U,B)\subset U$, we conclude that $y=\Phi_W(U,b)$. Hence
$$\|y-\Phi_W(V,a)\|=\|\Phi_W(U,b)-\Phi_W(V,a)\|<\eta,$$
and therefore
\begin{align*}
    d\big(x,\Lambda_W(V, A)\big)&\leq \|x-y\|+d\big(y, \Lambda_W(V,A)\big)\\
    &\leq \|x-y\|+\|y-\Phi_W(V,a)\|\\
    &<\|x-y\|+\eta.
\end{align*}

 Since the last inequality holds for every $y\in \Lambda_W(U,B)\cap B(0,L)$, we can then use Claim 2 to conclude that
 \begin{align*}
    d\big(x,\Lambda_W(V,A)\big)&\leq d\big(x, \Lambda_W(U,B)\cap B(0,L)\big)+\eta\\
     &=d\big(x, \Lambda_W(U,B)\big)+\eta.
 \end{align*}

 Thus,  $d\big(x,\Lambda_W(V,A)\big)-d\big(x, \Lambda_W(U,B)\big)<\eta$.
Symmetrically, we can prove that $d\big(x,\Lambda_W(U,B)\big)-d\big(x, \Lambda_W(V,A)\big)<\eta$, and therefore
$$\vert d\big(x,\Lambda_W(V,A)\big)-d\big(x, \Lambda_W(U,B)\big)\vert \leq \eta.$$
Since this is true for every $x\in B(0,j)$, we infer that 
$$\sup_{\|x\|< j}\{|d\big(x,\Lambda_W(V,A)\big)-d\big(x, \Lambda_W(U,B)\big)|\}\leq\eta <1/j.$$
By Lemma~\ref{lem:Lemma-Ananda}, we conclude that
$d_{AW}\big(\Lambda_W(V,A), \Lambda_W(U, B)\big)<1/j\leq \varepsilon$, and now the proof is complete. 
\end{proof}

\begin{theorem}\label{t:main K_kn es haz fibrado}
Define $\Lambda: \mathcal K_k^n\to \Gr_k(n)$ as
\begin{equation}\label{eq: map Lambda}
    \Lambda (A):=q\big(\Phi(A)\big)
\end{equation}
  where $q$ and  $\Phi$ are the continuous maps defined in (\ref{eq:traslacion de F}) and (\ref{eq: map Phi}), respectively.
\begin{enumerate}
    \item $\Lambda$ is a retraction.
    \item For every $W\in \Gr_k(n)$, the set $\Lambda^{-1}(\widetilde W)$ is homeomorphic with the topological product $\widetilde W\times \mathbb R^{n-k}\times\mathcal K^k_k$.
    \item The tuple $(\mathcal K_n^k, \Gr_k(n), \mathbb R^{n-k}\times\mathcal K^k_k, \Lambda)$ is a fiber bundle.
    \item If $\Gamma:\mathcal K_{k,b}^n\to \Gr_k(n)$ denotes the restriction $\Lambda|_{\mathcal {K}_{k,b}^n}$, then $\Gamma^{-1}(\widetilde W)$ is homeomorphic with the topological product $\widetilde W\times\mathbb R^{n-k}\times  \mathcal K_{k,b}^k$ and the tuple $(\mathcal K_n^k, \Gr_k(n), \mathbb R^{n-k}\times  \mathcal K_{k,b}^k, \Gamma)$ is a fiber bundle.
\end{enumerate}
\end{theorem}

\begin{proof}
    Since $\Phi$ and $q$ are retractions, the map $\Lambda$ is a retraction. 
Fix an element $W\in \Gr_k(n)$. Since $W$ is linearly isomorphic with $\mathbb R^k$, it is clear that $\mathcal K_k^k$ is homeomorphic with $\mathcal K_k^W$. Similarly,
    $W^\perp$ is linearly isomorphic with $\mathbb R^{n-k}$. Thus, in order to prove $(2)$ and $(3)$, it is enough to exhibit a homeomorphism $\psi_W:\widetilde W\times W^\perp\times\mathcal K_k^W\to \Lambda^{-1}(\widetilde W)$ such that 
    \begin{equation}\label{eq:lo que debe cumplir Psi}
        \Lambda\big(\psi_W(V,\omega,A)\big)=V, \;\text{for all }(V,\omega,A)\in \widetilde W\times W^\perp\times\mathcal K_k^W.
    \end{equation}

To that end, consider the continuous map $\Lambda_W$ from Lemma~\ref{l:w_V inducida es continua}, and define $\psi_W: \widetilde W\times W^\perp\times\mathcal K_k^W\to \Lambda^{-1}(\widetilde W)$ as
$$\psi_W(V,\omega, A):=\Lambda_W(V, A)+\omega.$$
By Proposition~\ref{p: basic properties} and the continuity of $\Lambda_W$, the map $\psi_W$ is continuous. 

If $(V,\omega,A)\in \widetilde W\times W^\perp\times\mathcal K_k^W $,
since $\Lambda_W(V, A)+\omega\subset V+\omega$, it is clear that $\Phi(\Lambda_W(V, A)+\omega)=V+\omega$ and therefore
\begin{align*}
\Lambda\big(\psi_W(V,\omega,A)\big)=q\big(\Phi(\Lambda_W(V, A)+\omega)\big)=q(V+\omega)=V\in \widetilde W.
\end{align*}
This implies that the map $\psi_W$ satisfies equation~(\ref{eq:lo que debe cumplir Psi}) and $\psi_W(V,\omega, A)$ effectively lies in $\Lambda^{-1}(\widetilde W)$.

It rests us to prove that $\psi_W$ is bijective and the inverse is continuous. To do this, consider the homeomorphism $\varphi_W$ defined in equation~(\ref{eq.definicion varphi W}).

To see that $\psi_W$ is injective, assume that $(V,\omega,A), (V',\omega',A')\in \widetilde W\times W^\perp\times\mathcal K_k^W $ satisfy 
$\psi_W(V,\omega,A)=\psi_W(V',\omega',A').$
Then
\begin{align*}
\varphi_W(V,\omega)&=V+\omega=\Phi\big(\psi_W(V,\omega,A)\big)\\
&=\Phi\big(\psi_W(V',\omega',A')\big)=V'+\omega'\\
&=\varphi_W(V',\omega').
\end{align*}

Since $\varphi$ is injective, we conclude that $V=V'$ and $\omega=\omega'$. Hence
$\Lambda_W(V,A)=\psi_W(V, \omega, A)-\omega=\psi_W(V, \omega, A')-\omega=\Lambda_W(V,A')$.
We can then use equality~(\ref{eq: proyeccion de lambda  es A}) to conclude that
$$A=\pi_W\big(\Lambda_W(V,A)\big)=\pi_W\big(\Lambda_W(V,A')\big)=A'.$$
This last equality implies that $\psi_W$ is injective. 

To prove that $\psi_W$ is surjective, recall that the inverse of $\varphi_W$ is given by 
$$\varphi_W^{-1}(F)=\Big(q(F), p(F)-\Phi_W\big(q(F),\pi_{W}(p(F)) \big)\Big)$$
 where  $p:\mathcal K^n\to\mathbb R^n$  is the map defined in equation~(\ref{eq:p}) and $\Phi_W$ is the map from Lemma~\ref{l:w_V es continua}. 
Thus, for every  $B\in\Lambda^{-1}(\widetilde W)$, 
there exists $(V,\omega)\in\widetilde W\times W^\perp$ such that
$$\varphi_W^{-1}(\Phi(B))=(V,\omega).$$
From here we infer that $V+\omega=\varphi_W(V,\omega)=\Phi(B)$ and therefore $V=\Phi(B)-\omega.$
Since $B\subset \Phi(B)$, we conclude that $B-\omega\subset V$. Define
$$A:=\pi_W(B-\omega).$$
Clearly $A\in\mathcal K^W_k$ and $\Lambda_W(V,A)=B-\omega$. Thus, 
\begin{align*}
    \psi_W(V, \omega, A)=\Lambda_W(V, A)+\omega=B-\omega+\omega=B.
\end{align*}
This last equality implies that $\psi_W$ is surjective. 
Furthermore, we have an expression for $\psi_W^{-1}$.
Indeed, let us  denote the second coordinate of the map $\varphi_W^{-1}$ by $f$,  namely
$$f(F):=p(F)-\Phi_W\big(q(F),\pi_{W}(p(F)) \big),\quad F\in q^{-1}(\widetilde W),$$
then 
 $\psi_W^{-1}:\Lambda^{-1}(\widetilde W)\to \widetilde W\times W^\perp\times \mathcal K_k^W$ can be written as
 $$\psi_W^{-1}(B):=\big(q(\Phi(B)), f(\Phi(B)), \pi_W(B-f(B))\big).$$

Since all maps involved in the expression of $\psi_W^{-1}$ are continuous, we conclude that $\psi_W^{-1}$ is continuous. This completes the proof of (2) and (3).

It only rests us to prove (4). However, this follows directly from (2), (3), and Lemma~\ref{l:w_V inducida es continua}-(2). 
    
\end{proof}

As we observed in Proposition~\ref{p: topologia de Knn}, the hyperspace $\mathcal K^{k}_k$ is a $Q$-manifold. Since $\Gr_k(n)$ is a manifold, every open subset in $\Gr_k(n)$  is an $\mathrm{ANR}$. In particular $\widetilde W$ and $\widetilde W\times \mathbb R^{n-k}$ are $\mathrm{ANR}$ for every $W\in \Gr_k(n)$.
This, in combination with Theorem~\ref{thm: X times Q}, implies that $\Lambda^{-1}(\widetilde W)\cong \widetilde W\times\mathbb R^{n-k}\times \mathcal K_k^k$ is a $Q$-manifold. Since the family
$\{\Lambda^{-1}(\widetilde W)\}_{W\in\Gr_k(n)}$ is an open cover of $\mathcal K_k^n$, we can use Remark~\ref{r: abiertos en Q manifolds} to conclude the following.

\begin{corollary}\label{c:K_kn es Q manifold}
If $0<k\leq n$, 
    $\mathcal K^{n}_k$ is a $Q$-manifold.
\end{corollary}

In the case of $\mathcal K^{n}_{k,b}$ we can be a little more specific.

\begin{corollary}\label{c:K_kbn es Q manifold}
Let $0<k< n$.
The hyperspace $\mathcal K_{k,b}^n$ is a $Q$-manifold and the map     $\Gamma:\mathcal K_{k,b}^n\to \Gr_k(n)$ 
from Theorem~\ref{t:main K_kn es haz fibrado}-4 is a fiber bundle projection with fiber $\mathbb R^{\frac{k(k+1)+2n}{2}}\times Q$.

\end{corollary}
\begin{proof}
By Theorem~\ref{t:main K_kn es haz fibrado}-4, 
the tuple $(\mathcal K_n^k, \Gr_k(n), \mathbb R^{n-k}\times  \mathcal K_{k,b}^k, \Gamma)$ is a fiber bundle. However, by Theorem~\ref{th:cb(n)}, $\mathcal K_{k,b}^k$ is homeomorphic with
$\mathbb R^{\frac{k(k+3)}{2}} \times Q$. Hence
$$\mathbb R^{n-k}\times  \mathcal K_{k,b}^k\cong \mathbb R^{n-k}\times\mathbb R^{\frac{k(k+3)}{2}}\times Q\cong \mathbb R^{\frac{k(k+1)+2n}{2}}\times Q.$$

In particular, for every $W\in\Gr_k(n)$,
$$\Gamma^{-1}(\widetilde W)\cong \widetilde W\times \mathbb R^{n-k}\times  \mathcal K_{k,b}^k\cong \widetilde W\times \mathbb R^{\frac{k(k+1)+2n}{2}}\times Q.$$
Since $\widetilde W\times \mathbb R^{\frac{k(k+1)+2n}{2}}$ is an $\mathrm{ANR}$, Theorem~\ref{thm: X times Q} implies that $\Gamma^{-1}(\widetilde W)$ is a $Q$-manifold. We can now use Remark~\ref{r: abiertos en Q manifolds} and the fact that $$\{\Gamma^{-1}(\widetilde W)\}_{W\in \Gr_k(n)}$$ is an open cover of $\mathcal K_{k,b}^n$ to conclude that 
$\mathcal K_{k,b}^n$  is a $Q$-manifold, as desired. 
\end{proof}


\begin{thebibliography}{99}


\bibitem{Alexandrino Bettiol} M. M. Alexandrino and R. G. Bettiol, \emph{Lie Groups and Geometric Aspects of Isometric Actions}, Springer, 2015.

\bibitem{Aliprantis} C. D. Aliprantis and K. C. Border, \emph{Infinite Dimensional Analysis: A Hitchhiker's Guide,} 3rd ed. Springer Berlin Heidelberg, New York, 2006.

\bibitem{AntonyanNatalia} S. Antonyan, N. Jonard-P\'erez, {\it Affine group acting on hyperspaces of compact convex subsets of $\mathbb{R}^n$}, Fundamenta Mathematicae 223 (2013), 99-136.

\bibitem{AntonyanNataliaSaulwidth} S. Antonyan, N. Jonard-P\'erez, and S. Ju\'arez-Ordóñez, \emph{Hyperspaces of convex bodies of constant width}, Topology and its Applications 196 (2015) 347-361.


\bibitem{Baker} A. Baker, \emph{Matrix Groups: An Introduction to Lie Group Theory}, Springer undergaduate mathematics series, 2001.

\bibitem{Barov} S. T. Barov, \emph{Smooth convex bodies in $\mathbb R^n$ with dense union of facets}, Topology proceedings  58 (2021) 71-83.

\bibitem{Bazylevych-93} L. E. Bazylevych, \emph{On the hyperspace of strictly convex bodies}, Mat. Studii 2 (1993) 83-86.

\bibitem{Bazylevych-97}  L. E. Bazylevych, \emph{Topology of a hyperspace of convex bodies of constant width}, Mat. Zametki 62 (6)
(1997) 813–819.

\bibitem{Bazylevych Zarichnyi-2006} L. E. Bazylevych, and M. M. Zarichnyi, \emph{On convex bodies of constant width} Topology Appl. 153 (11)
(2006) 1699–1704.



\bibitem{Beer1993} G. Beer, {\it Topologies on Closed and Closed Convex Sets}, MAIA 268, Kluwer Acad. Publ., 1993.

\bibitem{Belegradek} I. Belegradek,  \emph{Hyperspaces of smooth convex bodies up to congruence}, Advances in Mathematics 332 (2018) 176-198.


\bibitem{Berkson} E. Berkson, \emph{Some metrics on the subspaces of a Banach Space}, Pac. J. Math. 13 (1) (1963), 7-22. 

\bibitem{Chapmanbook} T. A. Chapman, {\it Lectures on Hilbert Cube Manifolds}, C. B. M. S. Regional Conference Series in Math., 28, Amer. Math. Soc., Providence, RI, 1975.

\bibitem{DonJuanJonardPerezLopezPoo} V. Donju\'an, N. Jonard-P\'erez and A. L\'opez-Poo, {\it Some notes on induced functions and group actions on hyperspaces}, Topol. Appl. 311 (2022) 107954.



\bibitem{Fabian et al}
M. Fabian,  P. Habala,  P. Hájek, V. Montesinos and V. Zizler, \emph{Banach Space Theory: The Basis for Linear and Nonlinear Analysis}, Springer, New York, 2011.

\bibitem{G-M-Gap}
I. C. Gokhberg, A. S. Markus, \emph{Two theorems on the gap between subspaces of a Banach space}, Usp. Mat. Nauk. 14 (5)
(1959) 135–140 (Russian).

\bibitem{Hatcher} A. Hatcher,  \textit{Algebraic Topology}, Cambridge University Press, 2002.

\bibitem{Hetman}  I. V. Hetman, \emph{Topological classification of the hyperspaces of polyhedral convex sets in normed spaces}, Matematychni Studii 39  (2) (2013) 203-211.



\bibitem{JonardMerino} B. Gonz\'alez-Merino and  N. Jonard-P\'erez, {\it{A pseudometric invariant under similarities in the hyperspace of non-degenerated compact convex sets of $\mathbb{R}^n$}} , Topology and its Applications 194 (2015) 125-143.

\bibitem{Luisa-Nat1} L. F. Higueras-Monta\~no and N. Jonard-P\'erez, \emph{On the topology of some hyperspaces of convex bodies associated to tensor norms}, Journal of Mathematical Analysis and Applications 509 (2) (2022) 125934.

\bibitem{Luisa-Nat} L. F. Higueras-Monta\~no and N. Jonard-P\'erez, \emph{A topological insight into the polar involution of convex sets}, Israel Journal of Mathematics (2024) 1-38. https://doi.org/10.1007/s11856-024-2622-0


\bibitem{Kadets}
M. I. Kadets,  \emph{Note on the gap between subspaces}, Functional Analysis and Its Applications 9 (2) (1975) 156-157.

\bibitem{Michael} E.
Michael,  \emph{Topologies on spaces of subsets}, Transactions of the American Mathematical Society 71 (1) (1951) 152-182.

\bibitem{Van Mill} J. van Mill, {\it Infinite-Dimensional Topology: Prerequisites and Introduction}, North-Holland Math. Library 43, Amsterdam, 1989.

\bibitem{Milnor Stasheff} J. W. Milnor and J. D. Stasheff, \emph{Characteristic Classes}, Princeton University Press, Princeton, New Jersey, 1974. 

\bibitem{Nadler} S. B. Nadler, Jr., J. E. Quinn, and N. M. Stavrakas, \emph{Hyperspaces of compact convex sets}, Pacific J. Math. 83 (1979), 441-462.

\bibitem{Ostrovskii} M. I.
Ostrovskii, \emph{Banach-Saks properties, injectivity and gaps between subspaces of a Banach space}, Journal of Soviet Mathematics 48 (3) (1990)  299-306.

\bibitem{Resende-Santos} P. Resende, J. P. Santos, {\it Open quotients of trivial vector bundles}, Topology and Applications 224 (2017) 19-47.

\bibitem{Sakai Bounded} K. Sakai,  \emph{The spaces of compact convex sets and bounded closed convex sets in a Banach space}, Houston J. Math. 34 (1) (2008) 289-300.

\bibitem{KSL} K. Sakai, {\it Geometric Aspects of General Topology}, Springer, Japan, 2013.



\bibitem{SakaiYaguchi2006} K. Sakai and M. Yaguchi, {\it The AR-property of the spaces of closed convex sets}, Colloquium Mathematicum, Vol. 106, No. 1 (2006), 15-24.



\bibitem{Sakai-Yang-2007} K. Sakai, Z. Yang, {\it The spaces of closed convex sets in euclidean spaces with the Fell topology}, Bulletin of The Polish Academy of Sciences Mathematics 55 (2007) 139-143.

\bibitem{Spanier} E. H. Spanier, \textit{Algebraic Topology}, Springer-Verlag, New york, 1966.

\bibitem{Steenrod} N. Steenrod, \textit{The Topology of Fibre Bundles}, Princeton University Press, USA, 1951.

\end{thebibliography}
\end{document}